\documentclass[reqno]{amsart}
\usepackage{amsbsy,amssymb,amscd,amsfonts,latexsym,amstext,delarray,
amsmath,mathrsfs,epsfig,graphicx}
\newenvironment{nitemize}
{\begin{list}{$\bullet$}
{\setlength{\parsep}{1ex}
\setlength{\topsep}{1.1ex}
\setlength{\partopsep}{0ex}
\setlength{\labelwidth}{0.5cm}
\setlength{\itemindent}{0cm}
\setlength{\itemsep}{0cm}
\setlength{\leftmargin}{0.7cm}
\setlength{\labelsep}{0.2cm}}}
{\end{list}}

\newenvironment{proof-theo}{\noindent{\it Proof of Theorem \ref{comparison-theorem}. }}{\fbox{}\\}

\newenvironment{proof-theo-2}{\noindent{\it Proof of Theorem \ref{use-Schauder-basis-to-get-V}. }}{\fbox{}\\}

\newenvironment{proof-corollary}{\noindent{\it Proof of Corollary \ref{criterium-dense-vector-space-unique-equilibrium}. }}{\fbox{}\\}

\newtheorem{theorem}{Theorem}

\newtheorem{corollary}{Corollary}
\newtheorem{lemma}{Lemma}

\theoremstyle{remark}
\newtheorem{remark}{Remark}

\theoremstyle{definition}

\theoremstyle{definition}
\newtheorem{example}{Example}

\def\N{{\mathbb N}}

\def\R{{\mathbb R}}

\def\Z{\mathbb Z}

\def\cN{\mathcal{N}}

\def\cV{\mathcal{V}}

\newcommand{\sE}{\mathscr{E}}

\newcommand{\sP}{\mathscr{P}}
\newcommand{\sQ}{\mathscr{Q}}

\newcommand{\ie}{{\it i.e.\/}\ }
\newcommand{\eg}{{\it e.g.\/}\ }
\newcommand{\cf}{{\it cf.\/}\ }

\def\text{\hbox}

\title[]{Criteria for the density of the graph of the entropy map restricted to ergodic states}

\begin{document}

\author{Henri Comman$^\dag$}
\address{Pontificia Universidad Cat\'{o}lica de Valparaiso,  Avenida Brasil 2950, Valparaiso, Chile}
\email{henri.comman@pucv.cl}
\thanks{$\dag$ Partially supported by FONDECYT grant 1120493.}



\begin{abstract}
We consider a non-uniquely ergodic  dynamical system given by a    $\mathbb{Z}^{l}$-action  (or $(\N\cup\{0\})^l$-action)  $\tau$ on a non-empty compact metrisable space    $\Omega$,  for some $l\in\N$.   Let  (D) denote  the following property:  The graph of the restriction of the  entropy map   $h^\tau$   to the set  of  ergodic states  is dense in the graph of $h^\tau$. We assume that   $h^\tau$ is  finite and upper semi-continuous. We give several criteria in order that  (D) holds,  each of which  is stated  in terms of a basic notion:   Gateaux differentiability of the pressure  map $P^\tau$ on 
  some   sets    dense in the space  $C(\Omega)$ of  real-valued continuous functions on $\Omega$,   level-2 large deviation principle, level-1 large deviation principle, convexity properties of some  maps on $\R^n$  for all $n\in\N$.
The one involving the Gateaux  differentiability of $P^\tau$ is of particular relevance in the context of large deviations since it establishes  a clear comparison  with another well-known sufficient condition: We show that for each  non-empty $\sigma$-compact subset  $\Sigma$ of  $C(\Omega)$,  (D) is equivalent to the existence of an infinite dimensional vector space 
$V$  dense in $C(\Omega)$ 
  such that $f+g$ has a  unique equilibrium state for all   $(f,g)\in \Sigma\times V\setminus\{0\}$; any Schauder basis $(f_n)$ of $C(\Omega)$ whose linear span contains $\Sigma$ admits an arbitrary small perturbation $(h_n)$ so  that one can take   $V=\textnormal{span}(\{f_n+h_n: n\in\N\})$. 
Taking $\Sigma=\{0\}$,    the existence of an infinite dimensional  vector space dense in $C(\Omega)$ constituted by functions admitting a  unique equilibrium state is equivalent to (D) together with   the uniqueness of measure of  maximal entropy. 
  \end{abstract}


\maketitle

\section{Introduction}\label{Introduction}
Let $(\Omega,\tau)$ be a dynamical system  in the sense of  \cite{Ruelle-78} (\ie $\Omega$ is a non-empty compact metrizable space and  $\tau$  an action of   $\mathbb{Z}^{l}$ (resp.  $(\N\cup\{0\})^l$) on   $\Omega$  for some $l\in\N$. Let  $C(\Omega)$, $\mathcal{M}(\Omega)$,
$\mathcal{M}^{\tau}(\Omega)$, ${\sE}^\tau(\Omega)$, 
$h^\tau$, $P^\tau$ denote respectively the set of real-valued continuous functions on $\Omega$ endowed with the uniform topology, Borel  probability measures on $\Omega$ endowed with the weak-$^*$ topology, $\tau$-invariant elements of $\mathcal{M}(\Omega)$, ergodic elements of $\mathcal{M}^{\tau}(\Omega)$,  measure-theoretic entropy and pressure maps.  We assume that   $\mathcal{M}^{\tau}(\Omega)$ is not a singleton  and  $h^\tau$ is finite and  upper semi-continuous.

 In some basic  dynamical systems  as above (\eg full shifts) the set $\mathcal{M}^{\tau}(\Omega)$ fulfils a fundamental  density property: Not only  ${\sE}^\tau(\Omega)$ is dense in $\mathcal{M}^{\tau}(\Omega)$ (\ie  $\mathcal{M}^{\tau}(\Omega)$ is the Poulsen simplex) 
 but the set $\{\mu\in\sE^{\tau}(\Omega): h^\tau(\mu)>r\}$ is dense in the set $\{\mu\in\mathcal{M}^{\tau}(\Omega): h^\tau(\mu)>r\}$ for every real $r$; 
 thanks to the upper semi-continuity of $h^\tau$,  this is equivalent to the density of  the graph of the restricted map  ${h^{\tau}}_{\mid \sE^\tau(\Omega)}$  in the graph of $h^\tau$; the importance of this   property  has long been recognized,  \cf \cite{Israel-Phelps-84-MAGHSCAN-54},  \cite{Israel-86-CMP-106}, \cite{Sokal-(82)-CMP-86}
  (in particular,  it implies the nowhere  Frechet differentiability   of $P^\tau$); following  \cite{Israel-Phelps-84-MAGHSCAN-54}  let us  denote it by    (D). 
 Since  a  measure on $\Omega$ is ergodic if and only if it  is the unique equilibrium state for some element in $C(\Omega)$ (\cite{Phelps_Dynamics and Randomness(2002)Santiago}),   (D)  is also equivalent to  Property 5.1 of \cite{Comman(2009)NON22} which  turns out  to be  sufficient to ensure the large deviation principle for 
any net $(\nu_\alpha)$ of 
  Borel probability measures on  $\mathcal{M}(\Omega)$ and any net $(t_\alpha)$ of positive real numbers converging to zero fulfilling  for some (arbitrary)   $f\in C(\Omega)$,
   \begin{equation}\label{introduction-eq30}
   \forall g\in C(\Omega),\ \ \ \ \ \ \lim_\alpha t_\alpha\log\int_{\mathcal{M}(\Omega)}e^{t_\alpha^{-1}\widehat{g}(\mu)}\nu_\alpha(d\mu)=P^{\tau}(f+g)-P^{\tau}(f), 
   \end{equation}
   where $\widehat{g}(\mu)=\int_{\Omega}g(\xi)\mu(d\xi)$     (\cite{Comman(2009)NON22}, Theorem 5.2). Another well-known sufficient  condition  to get the large deviation principle for    nets $(\nu_\alpha,t_\alpha)$ fulfilling (\ref{introduction-eq30})  is the existence of a vector space $V$  dense in $C(\Omega)$ 
  such that $f+g$ has a  unique equilibrium state for all $g\in V$ (\cite{Kifer-TAMS-90}, \cite{Comman_Rivera-Letelier(2010)ETDS31}); note that  by taking $g=0$ this   implies the uniqueness of equilibrium for  $f$, whereas (D) does not impose any conditions on $f$.

  A basic problem    is to compare  the two above conditions:  Does one  imply the other? Are they equivalent? If not, which extra hypotheses have to be added to get an equivalence? We can also compare them with the large deviation property: Do the large deviation principles imply one (or both) of these conditions? If not, do exist simple  extra hypotheses on  the  rate function in order to get an  equivalence?
 The  same questions araise   for the net  $((\widehat{f}_1,...,\widehat{f}_n)[\nu_\alpha])$  image of  $(\nu_\alpha)$  by the  map $(\widehat{f}_1,...,\widehat{f}_n)$ for any  $((f_1,...,f_n),n)\in C(\Omega)^n\times\N$ (and more generally for any net  of Borel probability measures on $\R^n$       admitting the same limiting log-moment generating function as $((\widehat{f}_1,...,\widehat{f}_n)[\nu_\alpha])$).
       As long as  one is  only concerned by (D),  more than the  nature of the net  satisfying the large deviation principle, the most   relevant object is
         the   rate function; more specifically, some fine  convexity properties  of the rate function play a major role. This leads us to consider convexity properties of some  maps involving only the restriction of  $P^\tau$ to finite dimensional spaces (or equivalently, its dual version with $h^{\tau}$)
             as a new element of comparison.

       In this paper we answer  to the preceding questions, showing that the five  above properties  (\ie  (D), Gateaux differentiability of $P^\tau$,  large deviation principle on $\mathcal{M}(\Omega)$, large deviation principle on $\R^n$ for all $n\in\N$, convexity properties  of some   maps on $\R^n$ for all $n\in\N$) are in fact equivalent once specified   how they take place (Theorem \ref{comparison-theorem}); in particular, we obtain  a plain and simple comparison between the two above mentioned general sufficient conditions to get the large deviation principle for  nets $(\nu_\alpha,t_\alpha)$ as in  (\ref{introduction-eq30}):  (D) is equivalent to  the  Gateaux differentiability of $P^\tau$ on 
       an infinite dimensional  vector space $V$ dense in $C(\Omega)$ possibly excepting zero; furthermore, for each    Schauder basis $(f_n)$ of $C(\Omega)$ and for each  sequence $(\varepsilon_n)$ of positive real numbers converging to zero, there is a sequence $(h_n)$ in $C(\Omega)\setminus\{0\}$ with $\mid\mid h_n\mid\mid\le\varepsilon_n$ so that  one can take   $V=\textnormal{span}(\{f_n+h_n: n\in\N\})$; when such a space $V$ is obtained, a Schauder basis  may be used to get another vector space  linearly independent from $V$ whose direct sum with $V$ fulfils the same property as $V$; iterating this process gives rise to a new criterion for the validity of (D) (Theorem \ref{use-Schauder-basis-to-get-V}). When there is  a unique measure of maximal  entropy, the above conditions can be greatly simplified (Corollary \ref{criterium-dense-vector-space-unique-equilibrium}). As a by-product, the  large deviation results of \cite{Comman_Rivera-Letelier(2010)ETDS31} are both   generalized and strengthened (Example \ref{ex-TCE}).

       In the next section we review some basic notions of thermodynamic formalism,  large deviation theory and convex analysis.    The results are stated in Section \ref{Results}.   The  proofs are given in Section \ref{Proofs}. The proof of the main theorem uses   a particular  case of two results of Israel and Phelps (\cite{Israel-Phelps-84-MAGHSCAN-54}) that we recall in   Appendix A.

\section{Preliminaries}\label{Preliminaries}

\subsection{Thermodynamic formalism}\label{subsection-Thermodynamic formalism}
 Let $(\Omega,\tau)$ be a dynamical system as in \S \ref{Introduction}. Put  $\Lambda(a)=\{(x_1,...,x_l)\in(\N\cup\{0\})^l:x_i<a_i,1\le i\le l\}$ and let ${|\Lambda(a)|}$ denote the cardinality of $\Lambda(a)$ for all $a\in \N^l$. 
For each $\varepsilon>0$ and for each
 $a\in \N^l$ let $\Omega_{\varepsilon,a}$ be
 a maximal
$(\varepsilon,\Lambda(a))$-separated set. 
Order $\N^l$  lexicographically. 
Recall that  $P^\tau(g)$ is defined for each $g\in C(\Omega)$ by
 \begin{equation}\label{subsection-Thermodynamic formalism-eq10}
 P^{\tau}(g)=
\lim_{\varepsilon\rightarrow 0}\limsup_a\frac{1}{|\Lambda(a)|}\log\sum_{\xi\in
\Omega_{\varepsilon,a}}e^{\sum_{x\in\Lambda(a)}g(\tau^x\xi)},
\end{equation}
and fulfills
\begin{equation}\label{subsection-Thermodynamic formalism-eq20}
P^{\tau}(g)=
\lim_{\varepsilon\rightarrow 0}\liminf_a\frac{1}{|\Lambda(a)|}\log\sum_{\xi\in
\Omega_{\varepsilon,a}}e^{\sum_{x\in\Lambda(a)}g(\tau^x\xi)}=\sup_{\mu\in\mathcal{M}^{\tau}(\Omega)}\{\mu(g)+h^\tau(\mu)\}.
\end{equation}
(\cite{Ruelle-78}, \S 6.7, \S 6.12 and Exercise 2 p. 119 for $\Z^l$-action, \S 6.18 for $(\N\cup\{0\})^l$-action). 
Since $\mathcal{M}^{\tau}(\Omega)$ is compact and  $h^\tau$  is finite and upper semi-continuous, the above supremum is a maximum, and  each element  realizing this  maximum is called  an   equilibrium state for $g$. The map  $P^\tau$ is finite convex and continuous on $C(\Omega) $(\cite{Ruelle-78}, \S 6.8 and  \S 6.18).  The right hand side of (\ref{subsection-Thermodynamic formalism-eq10})   may be  called the topological pressure versus the  variational pressure appearing in  the right hand side of the last equality in (\ref{subsection-Thermodynamic formalism-eq20});    the equality between both quantities is known as the variational principle. 
The map $h^\tau$ is affine;   the set $\mathcal{M}^{\tau}(\Omega)$ is a non-empty metrizable Choquet simplex  and  ${\sE}^\tau(\Omega)$ is the set of extreme points of $\mathcal{M}^{\tau}(\Omega)$ (\cite{Ruelle-78}, \S 6.1,   \S 6.5 and \S 6.18).

A sequence $(f_n)$ in $C(\Omega)$ is a Schauder basis of $C(\Omega)$ if for each $f\in C(\Omega)$ there exists a unique sequence $(\lambda_n(f))$ of real  numbers such that $\lim_n\mid\mid f-\sum_{k=1}^n \lambda_k(f) f_k\mid\mid=0$.  It is known that $C(\Omega)$ admits a Schauder basis  (\cite{Semeradi-82}, Theorem 4.4.13 and Notes pp. 8-9); furthermore, each vector space dense in $C(\Omega)$ contains a Schauder basis of $C(\Omega)$ (\cite{Semeradi-82}, Corollary 1.1.9). 
  Each  Schauder basis $(f_n)$ of $C(\Omega)$  fulfils  
\[\forall n\in\N,\ \ \ \ \ \ \sup_{f\in C(\Omega), \mid\mid f\mid\mid\le 1}\mid \lambda_n(f)\mid <+\infty\] (\cite{Semeradi-82}, Proposition 1.1.6). We will use the following result: For each Schauder basis $(f_n)$ of $C(\Omega)$ and for each
 sequence $(h_n)$ in $C(\Omega)$ fulfilling  \[\sum_{n=1}^{+\infty}\left(\sup_{f\in C(\Omega), \mid\mid f\mid\mid\le 1}\mid \lambda_n(f)\mid\right)\mid\mid h_n\mid\mid<1,\] 
 the sequence $(f_n+h_n)$ is a Schauder basis of $C(\Omega)$ (\cite{Semeradi-82}, Theorem 1.1.8).

\subsection{Convex analysis}\label{subsection-Convex analysis}
Let $X$ be  a Hausdorff real topological vector space and   let 
 $X^*$ be the topological dual of $X$ endowed with the weak-$^*$ topology. Let $A$ be a nonempty convex subset of $X$ and let 
 $I$ be a $]-\infty,+\infty]$-valued  function on $A$. 
   The function $I$ is  convex if 
  \[\forall (x,y,\lambda)\in A^2\times[0,1],\ \ \ \ \ \ \ \ I(\lambda x+(1-\lambda)y)\le \lambda I(x)+(1-\lambda)I(y).\]
        The set of all $x\in A$ such that   $I(x)\in\R$ is called    the effective domain of $I$.  The function     $I$ is proper if  the effective domain of $I$  is nonempty;
     in this case, 
for each convex subset $C$ of the effective domain of $I$, $I$ is said to be strictly convex on $C$ if 
\[I(\lambda x+(1-\lambda)y)<\lambda I(x)+(1-\lambda)I(y)\]
                       for all  $(x,y,\lambda)\in C^2\times\ ]0,1[$ with $x\neq y$. 
                                   
            Let us assume furthermore  that $A=X$.
                  The Legendre-Fenchel transform (also called convex conjugate) $I^*$ of $I$ is the function defined on $X^*$ by
\[\forall u\in X^*,\ \ \ \ \ \ \ I^*(u)=\sup_{x\in X}\left\{u(x)-I(x)\right\};\]
note that $I^*$ is a  proper convex function on $X^*$ when $I$ is proper. An element  $u\in X^*$ is a subgradient of $I$ at $x\in X$ if 
 \[\forall y\in X,\ \ \ \ \ \ \ \ I(y)\ge I(x)+u(y-x);\]
  note that when $I$ is proper  the above inequality implies that $x$ belongs to the effective domain of $I$.   An element  $u\in X^*$ is a subgradient of $I$ at $x\in X$ 
 if and only if one of the following conditions holds:
 \begin{itemize}
 \item $I^*(u)+I(x)\le u(x)$;
 \item $I^*(u)+I(x)= u(x)$.
 \end{itemize}
If  $I$ is  lower semi-continuous, then $u\in X^*$ is a subgradient of $I$ at $x\in X$ if and only if $x\in X$ is a subgradient of $I^*$ at $u\in X^*$ (\cite{Ekeland_Teman}, Corollary 5.2).
For each $(x,y)\in X^2$ we put
\[dI(x;y)=\lim_{\varepsilon\rightarrow 0^+} \frac{I(x+\varepsilon y)-I(x)}{\varepsilon};\]
   $dI(x;y)$ is a well-defined element of the extended real line (by convexity) and  called the directional derivative of  $I$ at $x$ in the direction $y$; 
 $I$ is Gateaux differentiable at $x$ if there exists $u\in X^*$ such that 
\[\forall y\in X,\ \ \ \ \ \ \ \ dI(x;y)=u(y);\] 
such an element $u$ is unique and   called the Gateaux differential of $I$ at $x$.
When  furthermore  $X$ is locally convex we have  $I={I^{**}}_{\mid X}$ (\cite{Ekeland_Teman}, Proposition 3.1 and Proposition 4.1) and    the two following results   hold for all $x\in X$:  If  $I(x)\in\R$ and $I$ is continuous at $x$, then $I$ is Gateaux differentiable at $x$  if and only if $I$ has a unique subgradient  at $x$; in this case, this subgradient  is the Gateaux differential of $I$ at $x$ (\cite{Ekeland_Teman}, Proposition 5.3).

The  above  notions    will be applied in  a finite as well as   infinite  dimensional setting; in this latter case we shall consider $X=C(\Omega)$ and $I=P^\tau$.  Some results recalled  in \S \ref{subsection-Thermodynamic formalism} may then   be rephrased in terms of convex analysis: For each  $g\in X$,  an element  $\mu\in  X^*$ is a subgradient of  $P^\tau$ at $g$ if and only if $\mu$ is  an equilibrium state for $g$; in particular, $g$ admits  a unique equilibrium state if and only if $P^\tau$ is Gateaux differentiable at $g$. The variational principle asserts that the entropy map $h^\tau$ is the restriction to $\mathcal{M}^{\tau}(\Omega)$ of the Legendre-Fenchel transform of $P^{\tau}$.

In the finite dimensional setting we will  need  the   notion of essential differentiability. Let $n\in\N$.  We assume that  $X=\R^n$ and $I$ is proper.  Let $\textnormal{dom}\ \delta I$ denote the 
set of points  where  $I$ has a  subgradient;  note that $\textnormal{dom}\ \delta I\neq\emptyset$ when the effective domain is not a singleton  (\cf \cite{Rockafellar-70}, Theorem 23.4). Then, $I$ is said to be essentially strictly convex if $I$ is strictly convex on 
every convex subset  of  $\textnormal{dom}\ \delta I$.  The function $I$ may   be essentially  strictly convex but not strictly convex on its effective domain; on the other hand, $I$ may be strictly convex on the relative interior of its effective domain, but not  essentially strictly convex (\cf \cite{Rockafellar-70}).     It is known that when  $I$ is lower semi-continuous and  $I^*$  has effective domain  $\R^n$, then $I$ is  essentially strictly convex       if and only if $I^*$ is differentiable on $\R^n$ (\cite{Rockafellar-70}, Theorem 26.3).

\subsection{Large deviations}\label{subsection-Large deviations}
Let $(\nu_\alpha,t_\alpha)$ be a net  where  $\nu_\alpha$ is a
Borel probability measure on a Hausdorff regular  topological space 
 $X$,  $t_\alpha>0$  and $(t_\alpha)$   converges to zero. We say that $(\nu_\alpha)$ satisfies a large
deviation principle  with powers $(t_\alpha)$ if there exists a
$[0,+\infty]$-valued lower semi-continuous function $I$  on $X$ such
that
\[\limsup t_\alpha\log\nu_{\alpha}(F)\le-\inf_{x\in F}I(x)\le-\inf_{x\in G}I(x)\le
 \liminf t_\alpha\log\nu_\alpha(G)\]
for all closed sets $F\subset X$ and all open sets $G\subset X$ with $F\subset G$; such a  function $I$ is then unique  and called the rate function.

Assume
furthermore that $X$ is a real topological vector space with
topological dual $X^*$ endowed with the weak-$^*$ topology. The map $\overline{L}$ defined on $X^*$ by
\[\forall\lambda\in X^*,\ \ \ \ \ \ \ \ \overline{L}(\lambda)=\limsup_\alpha t_\alpha\log\int_{X}e^{t_\alpha^{-1}\lambda(x)}\nu_\alpha(dx)\]
is called the  generalized   log-moment generating function (associated with $(\nu_\alpha,t_\alpha)$);  it  is a $]-\infty,+\infty]$-valued proper convex function; when the above upper limit is a limit   it is called the limiting   log-moment generating function at $\lambda$.

 The net   $(\nu_\alpha)$ is said to be
exponentially tight with respect to $(t_\alpha)$
 if for each real $M$ there exists a compact set
 $K_M\subset X$  such that
\[\limsup_\alpha t_\alpha\log\nu_\alpha(X\verb'\'K_M)<M.\]
It is known that if  $(\nu_\alpha)$ is 
exponentially tight with respect to $(t_\alpha)$ and $\overline{L}$ is $\R$-valued, then  $\overline{L}$ is  weak-$^*$ lower semi-continuous;  when furthermore $X$ is locally convex and $(\nu_\alpha)$ satisfies a large
deviation principle with powers $(t_\alpha)$ and   convex rate function $I$, then $I=\overline{L}^*$, where $\overline{L}^*$ is   the   
Legendre-Fenchel transform of   $\overline{L}$ (\cite{Dembo-Zeitouni}, Theorem 4.5.10; \cite{Comman(2007)STAPRO77}, Corollary 2); in this case we have   $I^*=\overline{L}$ by the  lower semi-continuity of  $\overline{L}$. If  $X=\R^n$ for some $n\in\N$ and the  limiting  log-moment generating function exists and is real-valued and differentiable on $\R^n$, then  a G\"{a}rtner's  theorem  asserts that $(\nu_\alpha)$ satisfies a large
deviation principle with powers $(t_\alpha)$ and  rate function $\overline{L}^*$ (\cite{Gartner-TeorVerojatnostPrimenen-77};  \cite{Dembo-Zeitouni}, Theorem 2.3.6).

If  $(\nu_\alpha)$ satisfies a large
deviation principle with powers $(t_\alpha)$ and rate function $I$, then 
for each real-valued continuous function $\lambda$ on $X$ fulfilling  
\begin{equation}\label{subsection-Large deviations-eq40}
\lim_{M\rightarrow+\infty}\limsup_\alpha t_\alpha\log\int_{\{x\in X:\lambda(x)>M\}}e^{t_\alpha^{-1}\lambda(x)}\nu_\alpha(dx)=-\infty
\end{equation}
the limiting   log-moment generating function  at $\lambda$ exists and given by
\begin{equation}\label{subsection-Large deviations-eq60}
\overline{L}(\lambda)=\sup_{x\in X}\{\lambda (x)-I(x)\}.
\end{equation}
The above result is a particular case of the Varadhan's theorem whose   original statement   requires the compactness of the level sets $\{x\in X: \lambda(x)\le r\}$  for all $r\in\R$  (\cite{Dembo-Zeitouni}, Theorem 4.3.1); this hypothesis has been removed in \cite{com-TAMS-03}, Corollary 3.4.
When $I$ is convex  and $\lambda\in X^*$,  the equality (\ref{subsection-Large deviations-eq60}) can be written as   $\overline{L}(\lambda)=I^*(\lambda)$, where $I^*$ denotes the  Legendre-Fenchel transform  of $I$.

 Let $\widetilde{\mathcal{M}}(\Omega)$ denote  the set  of signed Radon measures on $\Omega$ endowed with the weak-$^*$ topology.
  The above notions will be applied with $X=\widetilde{\mathcal{M}}(\Omega)$,  $X=\mathcal{M}(\Omega)$   and 
  $X=\R^n$ for all $n\in\N$. Since  $\widetilde{\mathcal{M}}(\Omega)^*=\{\widehat{g}:g\in C(\Omega)\}$, 
  the equation (\ref{introduction-eq30}) means that  the limiting   log-moment generating function associated with   $(\nu_\alpha,t_\alpha)$ exists and coincides with the map   \[\widetilde{\mathcal{M}}(\Omega)^*\ni\widehat{g}\mapsto P^{\tau}(f+g)-P^{\tau}(f)\] (the net $(\nu_\alpha)$ is  thought of as a net of measures on $\widetilde{\mathcal{M}}(\Omega)$).  
 In all  cases, the exponential tightness holds: This is obvious when $X=\widetilde{\mathcal{M}}(\Omega)$  (resp.  $X=\mathcal{M}(\Omega)$)  since  $(\nu_\alpha)$ is supported by the compact set
  $\mathcal{M}(\Omega)$;  in particular, (\ref{subsection-Large deviations-eq40}) holds for all $\lambda\in X^*$; it is known that in such a situation the large deviation principle in $\widetilde{\mathcal{M}}(\Omega)$ with rate function $I$   is equivalent to the  large deviation principle in $\mathcal{M}(\Omega)$ with  rate function $I_{\mid \mathcal{M}(\Omega)}$; furthermore, $I$ takes the value $+\infty$ on $\widetilde{\mathcal{M}}(\Omega)\setminus\mathcal{M}(\Omega)$ (\cite{Dembo-Zeitouni}, Lemma 4.1.5).  When   $X=\R^n$,  the exponential tightness  follows from  the  finiteness of the limiting  log-moment generating function on $\R^n$, which will always be the case with the nets  we shall consider.  In the above setting, a large deviation principle in $\mathcal{M}(\Omega)$  (resp. $\R^n$) is  commonly referred as  level-2 (resp. level-1).

         The  result of convex analysis recalled  in the last sentence of \S \ref{subsection-Convex analysis} will be applied in particular with $I^*$ the limiting   log-moment generating function associated with
                        $((\widehat{f}_1,...,\widehat{f}_n)[\nu_\alpha]),t_\alpha)$, where $(\nu_\alpha,t_\alpha)$ fulfils (\ref{introduction-eq30}) 
              and $f_1,...,f_n$ are suitable elements of $C(\Omega)$; the function $I$  will be the rate function governing the large deviation principle.

\subsection{Linking  large deviations with  thermodynamic  formalism by convex analysis}\label{subsection-relating-2.1-2.2-2.3}
              Given $f\in C(\Omega)$,            the relation   (\ref{introduction-eq30}) is a crucial equality that   not only  makes the bridge between the large deviation theory  and  thermodynamic  formalism by relating   the limiting log-moment generating function $L_f$ associated with $(\nu_\alpha,t_\alpha)$ to the pressure function $P^\tau$,   but  when   furthermore $(\nu_\alpha)$ satisfies a large
deviation principle in $\mathcal{M}(\Omega)$   with powers $(t_\alpha)$  and convex rate function $I_f$, it  allows to express  most basic ingredients of thermodynamic  formalism in terms of   $I_f$;  in particular,  (D) can be formulated in terms of a well-known  sufficient condition  on a convex rate function to get the large deviation principle (\cf \cite{Comman(2009)NON22}, Theorem 2.1, Property 5.1 and Theorem 5.2). Indeed, in this case,  $I_f={L_f^*}_{\mid \mathcal{M}(\Omega)}$ and the net  $(\nu_\alpha)$ satisfies a large
deviation principle in $\widetilde{\mathcal{M}}(\Omega)$   with powers $(t_\alpha)$  and  rate function $\widetilde{I_f}=L_f^*$ (\cf \S \ref{subsection-Large deviations}); since 
\[\forall \mu\in\widetilde{\mathcal{M}}(\Omega),\ \ \ \ \ \ \ \ \  L_f^*(\mu)=\sup_{\widehat{g}\in{\widetilde{\mathcal{M}}(\Omega)}^*}\{\widehat{g}(\mu)-L_f(\widehat{g})\}=\sup_{g\in C(\Omega)}\{\mu(g)-Q_f(g)\}=Q_f^*(\mu),\]
where $Q_f$ is  the map defined on $C(\Omega$ by 
\[\forall g\in C(\Omega),\ \ \ \ \ \ \ \ \ Q_f(g)=P^{\tau}(f+g)-P^{\tau}(f), \]
we get (\cf Lemma \ref{Fenchel-Legendre-transform-of-Q_f}),
\[
\forall\mu\in\widetilde{\mathcal{M}}(\Omega),\ \ \ \ \ \ \ \  \widetilde{I_f}(\mu)=
\left\{
\begin{array}{ll}
P^{\tau}(f)-h^{\tau}(\mu)-\mu(f) & \textnormal{if $\mu\in\mathcal{M}^{\tau}(\Omega)$}
\\ 
+\infty & \textnormal{if $\mu\in{\widetilde{\mathcal{M}}}(\Omega)\setminus \mathcal{M}^{\tau}(\Omega)$}.
\end{array}
\right.
\]           
                      Therefore, $\mathcal{M}^{\tau}(\Omega)$ is the effective domain of $I_f$; the entropy map $h^\tau$ coincides with ${I_f}_{\mid  \mathcal{M}^{\tau}(\Omega)}$ modulo an affine  function;  given a measure $\mu\in\mathcal{M}(\Omega)$, $\mu$ is an  equilibrium  state  for $f$ if and only if $I_f(\mu)=0$,  $\mu$ is ergodic if and only if $\mu$ is  the unique zero of $I_f$ for some $f\in C(\Omega)$;  for every  net $(\mu_i)$ of ergodic measures,                     
           $\lim_i (\mu_i,h^{\tau}(\mu_i))=(\mu,h^{\tau}(\mu))$  if and only if  $\lim_i (\mu_i,I_f(\mu_i))=(\mu,I_f(\mu))$. 
           It follows from the above correspondences that 
                   (D) is equivalent to   Property 5.1 of \cite{Comman(2009)NON22}, which  is nothing but a particular case of 
                   the 
                     condition appearing in Baldi's theorem in large deviation theory (\cf \cite{Comman(2009)NON22}, Theorem 2.1 and the proof of Theorem 3.3).

                       It is worth noticing that              (\ref{introduction-eq30})            is also necessary in order that     $(\nu_\alpha)$ satisfy the  large deviation principle in $\mathcal{M}(\Omega)$  (resp.  $\widetilde{\mathcal{M}}(\Omega)$) with powers $(t_\alpha)$  and the above   convex rate function $I_f$ (resp. $\widetilde{I_f}$); this is a consequence of Varadhan's theorem; indeed, in this case, since (\ref{subsection-Large deviations-eq40}) holds for all $\lambda\in X^*$  with $X=\widetilde{\mathcal{M}}(\Omega)$, 
                  the limiting log-moment generating function $L_f$ associated with $(\nu_\alpha,t_\alpha)$  exists   as a convex  lower semi-continuous  function  on $X$  and      fulfils  $L_f=\widetilde{I_f}^*$  by  (\ref{subsection-Large deviations-eq60}); since ${\widetilde{I_f}}=Q_f^*$ and $Q_f$ is continuous  we have $Q_f^{**}=Q_f$ hence 
                  $L_f(\widehat{g})=Q_f(g)$ for all $g\in C(\Omega)$, 
                  which is exactly   (\ref{introduction-eq30}).

\section{Results}\label{Results}

         As a natural candidate for  $(\nu_\alpha,t_\alpha)$ as in (\ref{introduction-eq30}),   for each     $f\in C(\Omega)$ we   introduce   a basic net $(\nu^\tau_{f,\alpha},t^\tau_\alpha)$      canonically  associated to the system $(\Omega,\tau)$: Let  $\wp$ denote    the product  set
$]0,+\infty[\times {\N^l}^{]0,+\infty[}$ pointwise directed, where $]0,+\infty[$ (resp. $\N$,  $\N^l$)  is endowed  with the inverse of the  natural order on $\R$ (resp. natural order, lexicographic order), \ie $(\varepsilon,u)\in\wp$ is less than or equal  $(\varepsilon',u')\in\wp$ if $\varepsilon\ge\varepsilon'$ and $u(\delta)$ is lexicographically  less than or equal $u'(\delta)$ for all $\delta\in\ ]0,+\infty[$ (\cf \cite{Kelley-91}). For each  $\alpha=(\varepsilon,u)\in\wp$ we put 
\[t^\tau_\alpha=\frac{1}{\mid\Lambda(u(\varepsilon))\mid}\] 
and
\[\forall f\in C(\Omega),\ \ \ \ \ \ \ \ \nu^\tau_{f,\alpha}=\sum_{\xi\in
\Omega_{\varepsilon,u(\varepsilon)}}\frac{e^{\sum_{x\in\Lambda(u(\varepsilon))}f(\tau^x\xi)}}{\sum_{\xi'\in
\Omega_{\varepsilon,u(\varepsilon)}}e^{\sum_{x\in\Lambda(u(\varepsilon))}f(\tau^x\xi')}}\delta_{\frac{1}{\mid\Lambda(u(\varepsilon))\mid}\sum_{x\in
 \Lambda(u(\varepsilon))}\delta_{\tau^x(\xi)}}.\] 
 The first equality in (\ref{subsection-Thermodynamic formalism-eq20}) implies  that  $(\nu^\tau_{f,\alpha}, t^\tau_\alpha)$ fulfils  (\ref{introduction-eq30}) (Lemma \ref{existence-net-generating the pressure}); this is obvious when $\tau$ is expansive, in which case the above net is in fact a sequence indexed by  elements of $\N^l$.

   Here is the main result.

\begin{theorem}\label{comparison-theorem}
Let $f_0\in C(\Omega)$, let $E$  be a  set generating a $\sigma$-compact    vector space $W$
  dense in $C(\Omega)$ and let $\Sigma$ be a nonempty   subset of  $W$. 

\begin{nitemize}
\item[a)] The following statements are equivalent:
\begin{itemize}
\item[$(i)$] $\{\mu\in\sE^{\tau}(\Omega): h^\tau(\mu)>r\}$ is dense in $\{\mu\in\mathcal{M}^{\tau}(\Omega): h^\tau(\mu)>r\}$ for all $r\in\R$; 
\item[(D)$\equiv (ii)$] The graph of $h^{\tau}_{\ \mid \sE^\tau(\Omega)}$ is dense in the graph of $h^{\tau}$.
\item[$(iii)$] For each   $n\in\N$   there exists    a set $D_n$  dense in  $C(\Omega)^n$ 
   such that   for each   $(g_1,
...,g_n)\in D_n$  and for each 
   $\varepsilon>0$ there 
   is 
 $g\in C(\Omega)\setminus\textnormal{span}\{g_1,...,g_n\}$ with $\mid\mid g\mid\mid<\varepsilon$
     such that  $g+\sum_{k=1}^n t_k g_k$ has a  unique equilibrium state for all $(t_1,...,t_n)\in\R^n$.
\item[$(iv)$] There exists an infinite dimensional     vector space $V$  dense in $C(\Omega)$    such that 
$f+g$ has a  unique equilibrium state for all   $(f,g)\in \Sigma\times V\setminus\{0\}$.
\item[$(v)$]    The net   $(\nu^\tau_{f_0,\alpha})$ satisfies a large deviation principle in $\mathcal{M}(\Omega)$ with powers $(t^{\tau}_\alpha)$ and a convex rate function  $I$ such that   the graph of   ${I}_{\mid \sE^\tau(\Omega)}$  is dense in the graph of  ${I}_{\mid \mathcal{M}^\tau(\Omega)}$.  
  \item[$(vi)$]  There exists  an infinite dimensional      vector space $V$  dense in $C(\Omega)$   such that for each $(f,g)\in \Sigma\times V\setminus\{0\}$ 
  the net   $(\nu^\tau_{f+g,\alpha})$ satisfies a large deviation principle in $\mathcal{M}(\Omega)$ with powers $(t^\tau_\alpha)$ and a  convex rate function  vanishing at a unique point.   
  \item[$(vii)$]  There exists  an infinite dimensional       vector space $V$ dense in  $C(\Omega)$   such that    the net 
    $((\widehat{f_1},...,\widehat{f_n})[\nu^\tau_{f+g,\alpha}])$ 
  satisfies a large deviation principle in $\R^n$ with powers $(t^\tau_\alpha)$ and  an essentially    strictly convex rate function  for all   $(f,g,(f_1,...,f_n), n)\in \Sigma\times V\setminus\{0\}\times E^n\times\N$. 
 \item[$(viii)$]    There exists  an infinite dimensional        vector space $V$ dense in  $C(\Omega)$   such that
  the map 
  \[\R^n\ni(x_1,...,x_n)\mapsto\sup_{(t_1,...,t_n)\in\R^n}\left\{\sum_{k=1}^n t_k x_k-P^\tau\left(f+g+\sum_{k=1}^n t_k f_k\right)\right\}\]
 is essentially strictly convex  for all    $(f,g,(f_1,...,f_n), n)\in \Sigma\times V\setminus\{0\}\times E^n\times\N$.
  \item[$(ix)$]  There exists  an infinite dimensional      vector space $V$ dense in  $C(\Omega)$   such that
  the map 
 \[\R^n\ni x\mapsto\inf\left\{-\mu(f+g)-h^{\tau}(\mu):\mu\in M^{\tau}(\Omega),(\mu(f_1),...,\mu(f_n))=x\right\}\]
  is essentially strictly convex for all  $(f,g,(f_1,...,f_n), n)\in \Sigma\times V\setminus\{0\}\times E^n\times\N$.
   \end{itemize}
   Furthermore: 
   \begin{itemize}
   \item[$1)$] If an infinite dimensional vector space $V$ dense in $C(\Omega)$ fulfils one of the conditions $(iv)$,  $(vi)$, $(vii)$, $(viii)$, $(ix)$, then $V$ fulfils all of them.
\item[$2)$]    The rate function in $(v)$ (resp. $(vi)$)  is 
   \[\mathcal{M}(\Omega)\ni\mu\mapsto 
\left\{
\begin{array}{ll}
P^{\tau}(f_0)-h^{\tau}(\mu)-\mu(f_0) & \textnormal{if $\mu\in\mathcal{M}^{\tau}(\Omega)$}
\\ 
+\infty & \textnormal{if $\mu\in\mathcal{M}(\Omega)\setminus \mathcal{M}^{\tau}(\Omega)$}
\end{array}
\right.
\]
\[\left(resp.\ \ \ \ \ \  \mathcal{M}(\Omega)\ni\mu\mapsto 
\left\{
\begin{array}{ll}
P^{\tau}(f+g)-h^{\tau}(\mu)-\mu(f+g) & \textnormal{if $\mu\in\mathcal{M}^{\tau}(\Omega)$}
\\ 
+\infty & \textnormal{if $\mu\in\mathcal{M}(\Omega)\setminus \mathcal{M}^{\tau}(\Omega)$}
\end{array}
\right.
\right),\]
and the rate function $I_{f+g,(f_1,...,f_n)}$ in $(vii)$ fulfils for each $(x_1,...,x_n)\in\R^n$, 
\begin{multline*}
I_{f+g,(f_1,...,f_n)}(x_1,...,x_n)-P^\tau(f+g)=\sup_{(t_1,...,t_n)\in\R^n}\left\{\sum_{k=1}^n t_k x_k-P^\tau\left(f+g+\sum_{k=1}^n t_k f_k\right)\right\}
\\
=\inf\left\{-\mu(f+g)-h^{\tau}(\mu):\mu\in M^{\tau}(\Omega),(\mu(f_1),...,\mu(f_n))=(x_1,...,x_n)\right\}.
\end{multline*}
\item[$3)$] In $(iii)$ one can take  $D_n=V^n$ with $V$ as in $(iv)$ (resp. $(vi)$, $(vii)$, $(viii)$, $(ix)$).    
\end{itemize}
\item[b)]  Part a)     holds verbatim with any one of the following changes:
      \begin{itemize}
    \item[$1)$] Replacing in $(v)$    the net  $(\nu^\tau_{f_0,\alpha})$ and $(t^\tau_\alpha)$ respectively   by any net  $(\nu_\alpha)$  of Borel probability measures on  $\mathcal{M}(\Omega)$ and any net  $(t_\alpha)$  of positive real numbers  converging to zero fulfilling  for each $h\in C(\Omega)$, 
   \[\lim_\alpha t_\alpha\log\int_{\mathcal{M}(\Omega)}e^{t_\alpha^{-1}\int_\Omega h(\omega)\mu(d\omega)}\nu_\alpha(d\mu)=P^\tau(f_0+h)-P^\tau(f_0);\]
     \item[$2)$] Replacing in $(vi)$ the net  $(\nu^\tau_{f+g,\alpha})$ and $(t^\tau_\alpha)$ respectively   by any net  $(\nu_\alpha)$  of Borel probability measures on  $\mathcal{M}(\Omega)$ and any net  $(t_\alpha)$  of positive real numbers  converging to zero fulfilling  for each $h\in C(\Omega)$, 
   \[\lim_\alpha t_\alpha\log\int_{\mathcal{M}(\Omega)}e^{t_\alpha^{-1}\int_\Omega h(\omega)\mu(d\omega)}\nu_\alpha(d\mu)=P^\tau(f+g+h)-P^\tau(f+g);\]
    \item[$3)$] Replacing in $(vii)$   the net $((\widehat{f_1},...,\widehat{f_n})[\nu^\tau_{f+g,\alpha}])$ 
 and $(t^\tau_\alpha)$ respectively   by any net  $(\mu_\alpha)$  of Borel probability measures on  $\R^n$ and any net  $(t_\alpha)$  of positive real numbers converging to zero fulfilling  for each $(t_1,...,t_n)\in\R^n$, 
   \[\lim_\alpha t_\alpha\log\int_{\R^n}e^{t_\alpha^{-1}\sum_{k=1}^n t_k x_k}\mu_\alpha(d(x_1,...,x_n))=P^\tau(f+g+\sum_{k=1}^n t_k f_k)-P^\tau(f+g).\]
   \item[$4)$] Requiring furthermore that $V\cap W=\{0\}$ in $(iv)$ (resp. $(vi)$, $(vii)$, $(viii)$, $(ix)$). 
   \end{itemize}
    \item[c)] If each element of $\Sigma$ has a unique equilibrium state, then part a) holds verbatim  replacing   $V\setminus\{0\}$ by $V$. 
         If  one of the  conditions $(iv)$,  $(vi)$, $(vii)$, $(viii)$, $(ix)$     holds   replacing   $V\setminus\{0\}$ by $V$, then  all  these  conditions  hold  replacing $V\setminus\{0\}$ by $V$ (in particular, each element of $\Sigma$ has a unique equilibrium state)
     and  
part b) remains true with this change.
    \end{nitemize}
   \end{theorem}

      When $W$ contains an element admitting several equilibrium states,   the proof of Theorem \ref{comparison-theorem} reveals that    $V$ is obtained as a proper subspace of a direct sum  $W\oplus\widetilde{W}$, where $\widetilde{W}$ is a $\sigma$-compact infinite dimensional   vector subspace of $C(\Omega)$ such that $f+h$ has a  unique equilibrium state for all   $(f,h)\in W\times (\widetilde{W}\setminus\{0\})$; the space  $\widetilde{W}$ is given by Lemma \ref{Israel-Phelps-theorem-strenght}   and  $V$ is obtained    by means of a  Schauder basis of $C(\Omega)$ included in $W$; 
       the existence of  such a  space $\widetilde{W}$   can in turn  be recovered  from Theorem \ref{comparison-theorem}:      The condition  $(iv)$ implies $V\cap W=\{0\}$ so that 
   $\widetilde{W}$ can be  taken  as any $\sigma$-compact infinite dimensional   vector subspace of  $V$.
    When each element of  $W$ has  a unique  equilibrium state, one considers  the space $W'=W\oplus\textnormal{span}(\{f\})$, where $f\in C(\Omega)$ has  several equilibrium states (such an element exists by Theorem 3.4 of \cite{Israel-Phelps-84-MAGHSCAN-54}, \cf Lemma \ref{existence-of-non-GD-element}) and the preceding case applies with $W'$ in place of $W$. Note that in all cases, 
   the space  $\widetilde{W}$ so obtained  may furthermore chosen as to be  dense in $V$, and thus dense in $C(\Omega)$, which  is   an extra   property  that is   not given by Lemma \ref{Israel-Phelps-theorem-strenght}  neither by  Theorem \ref{Israel-Phelps-theorem} in Appendix A on which 
          Lemma \ref{Israel-Phelps-theorem-strenght} is based.

          The following theorem  specifies  the nature of $V$ and  explains      the use of Schauder bases; 
           it  establishes a method that permits us from any $\sigma$-compact  space $V$ as above to get a new one $V_1$  linearly independent from  $V$ and  such that $V\oplus V_1$ fulfils the same properties as $V$; by iterating this process,  we obtain a
             infinite direct sum $\bigoplus_{n=0}^\infty V_n$ with $V_0=V$, whose      existence        furnishes a new criterion for the validity of (D);  furthermore, starting with $\bigoplus_{n=1}^\infty V_n$ and using only subspaces of this sum, 
            the above  method   allows us  to build
           another direct sum     $\bigoplus_{n=1}^\infty \widetilde{V}_n$ linearly independent from $\bigoplus_{n=1}^\infty V_n$ and fulfilling the same properties as $\bigoplus_{n=1}^\infty V_n$.                         
       
    \begin{theorem}\label{use-Schauder-basis-to-get-V}
    Let $E$, $\Sigma$ and $W$ be as in Theorem \ref{comparison-theorem}.
    Let  $V$ be  an infinite dimensional  vector  space  dense in $C(\Omega)$ fulfilling one of the conditions $(iv)$,  $(vi)$, $(vii)$, $(viii)$, $(ix)$ of Theorem \ref{comparison-theorem} with   $E$, $\Sigma$ and $W$. Let $V_0$ be a   $\sigma$-compact   vector  space dense in $V$.
   \begin{nitemize}
   \item[a)]      There exists  an infinite direct sum $\bigoplus_{n=1}^\infty V_n$, where each $V_n$ is a  
    $\sigma$-compact infinite dimensional vector space dense in $C(\Omega)$ and  linearly independent from $V_0$,   such that for each   nonempty subset $\cN$ of $\N\cup\{0\}$  the direct  sum $\bigoplus_{n\in\cN} V_n$  fulfils all the  conditions $(iv)$,   $(vi)$, $(vii)$, $(viii)$, $(ix)$ of Theorem  \ref{comparison-theorem} applied with $W+\bigoplus_{n\in\N\cup\{0\}\setminus\cN} V_n$ in place of $W$ (putting $\bigoplus_{n\in\emptyset} V_n=\{0\}$ when $\cN=\N\cup\{0\}$).
\item[b)] An infinite direct sum as  in part $a)$ may be  obtained  by recurrence  in the following way: 
  For each $n\in\N\cup\{0\}$,  given  $V_0,...,V_n$,  we consider a $\sigma$-compact  infinite dimensional  vector space $\widetilde{W_n}$  such that 
$f+g$ has a unique equilibrium state for all $(f,g)\in W+\bigoplus_{j=0}^{n} V_j\times(\widetilde{W_n}\setminus\{0\})$ and   we put
\[V_{n+1}=\textnormal{span}(\{f_{n,k} + h_{n,k}:k\in\N\}),
\]
 where $(f_{n,k})$ is a Schauder basis  of $C(\Omega)$ included in $W+\bigoplus_{j=0}^{n} V_j$ and $\{h_{n,k}:k\in\N\}$  a linearly independent subset  of      $\widetilde{W_{n}}$  fulfilling 
\begin{equation}\label{use-Schauder-basis-to-get-V-eq20}
\sum_{k=1}^{+\infty}\left(\sup_{f\in C(\Omega), \mid\mid f\mid\mid\le 1}\mid \lambda_{n,k}(f)\mid\right)\mid\mid h_{n,k}\mid\mid<1,
\end{equation}
where $(\lambda_{n,k}(f))$ denotes the coordinates of $f$ in the basis $(f_{n,k})$. 
 Furthermore, if  $V_0\cap W=\{0\}$  then the above assertion holds verbatim replacing $W+V_0$ by $W$ and taking $\widetilde{W_0}=V_0$; this is the case in particular   when $W$ contains an element admitting several equilibrium states. 

\item[c)] Let  $\bigoplus_{n=1}^\infty V_n$ be  as in part $a)$,  let $(m_n)$ be a strictly increasing sequence in $\N$ and  for each $n\in\N$  let   $\cV_n$ be a $\sigma$-compact infinite dimensional vector subspace of $\bigoplus_{j=m_n}^{j=m_{n+1}} V_j$. Put 
$\widetilde{V}_0=V_0$  and 
\[\forall n\in\N\cup\{0\},\ \ \ \ \ \ \ \ \widetilde{V}_{n+1}=\textnormal{span}(\{f_{n,k} + h_{n,k}:k\in\N\}),
\]
where $(f_{n,k})$ is a Schauder basis  of $C(\Omega)$ included in $W+\bigoplus_{j=0}^{n} \widetilde{V}_j$ and 
 $(h_{n,k})$   a linearly independent subset of $\cV_{n+1}$ fulfilling (\ref{use-Schauder-basis-to-get-V-eq20}). Then, the direct sum $\bigoplus_{n=1}^\infty \widetilde{V}_{n}$  is linearly independent from  $\bigoplus_{n=1}^\infty V_n$ and fulfils  the same  properties as $\bigoplus_{n=1}^\infty V_n$ stated  in part $a)$.  If $V_0\cap W=\{0\}$, then the above assertion holds verbatim replacing  $W+\widetilde{V}_0$ (resp. $\cV_1$) by $W$ (resp. $V_0$).  
  \end{nitemize}
    \end{theorem}

 The conclusions of part c) of Theorem \ref{comparison-theorem} with $\Sigma=\{0\}$  concerning the level-1 large deviations and the convex functions related with the associated  rate functions (namely, conditions $(vii)$, $(viii)$ and $(ix)$)    can be substantially improved as shows  the  following corollary.

\begin{corollary}\label{criterium-dense-vector-space-unique-equilibrium}
The following statements are equivalent:
\begin{itemize}
\item[$(i)$] There  exists  a unique measure of maximal entropy and the graph of $h^{\tau}_{\ \mid \sE^\tau(\Omega)}$ is dense in the graph of $h^{\tau}$.
\item[$(ii)$] There exists an infinite dimensional        vector space $V$ dense in $C(\Omega)$ such that each element of $V$ has a  unique equilibrium state;
\item[$(iii)$]  There exists   an infinite dimensional     vector space $V$ dense in  $C(\Omega)$    such that  the map 
 \[\R^n\ni(x_1,...,x_n)\mapsto\sup_{(t_1,...,t_n)\in\R^n}\left\{\sum_{k=1}^n t_k x_k-P^\tau\left(\sum_{k=1}^n t_k f_k\right)\right\}\]
 is essentially strictly convex  for all   $((f_1,...,f_n), n)\in V^n\times\N$.
\item[$(iv)$] There exists  an infinite dimensional     vector space $V$ dense in  $C(\Omega)$ 
  such that    the map 
 \[\R^n\ni x\mapsto\inf\left\{-h^{\tau}(\mu):\mu\in M^{\tau}(\Omega),(\mu(f_1),...,\mu(f_n))=x\right\}\]
  is essentially strictly convex for all   $((f_1,...,f_n), n)\in V^n\times\N$.
  \item[$(v)$]  There exists  an infinite dimensional       vector space $V$ dense in  $C(\Omega)$   such that    the net 
    $((\widehat{f_1},...,\widehat{f_n})[\nu^\tau_{0,\alpha}])$ 
  satisfies a large deviation principle in $\R^n$ with powers $(t^\tau_\alpha)$ and an essentially strictly convex
   rate function for all   $((f_1,...,f_n), n)\in V^n\times\N$.
  \end{itemize}
  Furthermore:
  \begin{itemize}
   \item[$1)$] If an infinite dimensional   vector  space $V$  dense in $C(\Omega)$ fulfils one of the  conditions $(ii)$, $(iii)$, $(iv)$, $(v)$,  then   for each $\sigma$-compact vector space $V_0$  dense in $V$     there exists  an infinite direct sum $\bigoplus_{n=1}^\infty V_n$, where each $V_n$ is a   $\sigma$-compact infinite dimensional vector space dense in $C(\Omega)$  and linearly independent from  $V_0$,  such that $\bigoplus_{n=0}^\infty V_n$ fulfils  all the   conditions  $(ii)$, $(iii)$, $(iv)$,  $(v)$. 
   \item[$2)$] The above equivalences    hold verbatim  replacing  in $(v)$ the net   $((\widehat{f_1},...,\widehat{f_n})[\nu^\tau_{0,\alpha}])$ and $(t^\tau_\alpha)$ respectively   by any net
  $(\mu_\alpha)$  of Borel probability measures on  $\R^n$ and any net  $(t_\alpha)$  of positive real numbers  converging to zero fulfilling for some (arbitrary and independent of $n$) $f\in V$ and for each $(t_1,...,t_n)\in\R^n$, 
     \[\lim_\alpha t_\alpha\log\int_{\R^n}e^{t_\alpha^{-1}\sum_{k=1}^n t_k x_k}\mu_\alpha(d(x_1,...,x_n))=P^\tau\left(f+\sum_{k=1}^n t_k f_k\right)-P^\tau(f).\]
    In this case, the rate function  $I_{f,(f_1,...,f_n)}$ governing the large deviation principle of   the above  net  (and in particular the net $((\widehat{f_1},...,\widehat{f_n})[\nu^\tau_{f,\alpha}])$ of $(v)$ with $f=0$)
     fulfils  for each $(x_1,...,x_n)\in\R^n$, 
\begin{multline*}
I_{f,(f_1,...,f_n)}(x_1,...,x_n)-P^\tau(f)=\sup_{(t_1,...,t_n)\in\R^n}\left\{\sum_{k=1}^n t_k x_k-P^\tau\left(f+\sum_{k=1}^n t_k f_k\right)\right\}
\\
=\inf\left\{-\mu(f)-h^{\tau}(\mu):\mu\in M^{\tau}(\Omega),(\mu(f_1),...,\mu(f_n))=(x_1,...,x_n)\right\}.
\end{multline*}
\end{itemize}
     \end{corollary}

\begin{remark}\label{remark-(D)-independence of Sigma}
The statements $(i)$, $(ii)$ and  $(iii)$ of  Theorem  \ref{comparison-theorem}    depend neither  on  $\Sigma$ nor on $E$ nor on $f_0$. Therefore,   Theorem  \ref{comparison-theorem} holds verbatim replacing  $\Sigma$ by any non-empty   subset  of  $W$ in $(iv)$ (resp. $(vi)$, $(vii)$, $(viii)$, $(ix)$); similarly, one can replace $E$ by any set generating $W$ in $(vii)$ (resp. $(viii)$, $(ix)$);  in other words, the sets $\Sigma$ and $E$ may be different in each condition where they appear.    The condition  $(v)$ holds for  some $f_0\in C(\Omega)$ if and only if $(v)$ holds for  all $f_0\in C(\Omega)$. 
\end{remark}

  \begin{remark}\label{usefulness of Schauder basis}
  The space $\widetilde{W_n}$ in part $b)$ of Theorem \ref{use-Schauder-basis-to-get-V}
   may be obtained by  Lemma \ref{Israel-Phelps-theorem-strenght} as well as by applying  Theorem  \ref{comparison-theorem} with  $W+\bigoplus_{j=0}^{n} V_j$ in place of $W$.
   In any case, the proof reveals that     given a linearly independent subset $\{h_k:k\in\N\}$  of $\widetilde{W_n}$, any set  $\{f_k:k\in\N\}\subset W+\bigoplus_{j=0}^{n} V_j$ such that $\textnormal{span}(\{f_k+h_k:k\in\N\})$ is dense in $C(\Omega)$ 
   would do (\cf first assertion of  Lemma \ref{condition on the h_n}). It turns out that  
   the concept of  Schauder basis    captures  more than  the properties needed: indeed,  when $(f_k)$ is a  Schauder basis of $C(\Omega)$,  there is a sequence $(\varepsilon_k)$ of positive real numbers converging to zero  such that 
      the sequence   $(f_k+h_k')$  is  a Schauder basis for every sequence $(h_k')$ satisfying
      \begin{equation}\label{usefulness of Schauder basis-eq20}
\forall k\in\N,\ \ \ \ \ \ \ \ \mid\mid h_k'\mid\mid\le\varepsilon_k
\end{equation} (\cf \S \ref{subsection-Thermodynamic formalism}); therefore, when $\mid\mid h_k\mid\mid\le\varepsilon_k$ for all $k\in\N$,  this  allows us   to  choose  not only      $(h_k')=(h_k)$,   but any linearly independent set 
  $\{h_k':k\in\N\}\subset\widetilde{W_n}$ fulfilling (\ref{usefulness of Schauder basis-eq20}); note that the linear independence of $\{h_k':k\in\N\}$ is necessary since  $\{f_k:k\in\N\}$ is linearly independent
  (\cf second assertion of Lemma \ref{condition on the h_n}). The same observation  holds regarding the proof of the implication $(i)\Rightarrow(iv_4)$ of  Theorem  \ref{comparison-theorem}. 
  \end{remark}

\begin{remark}\label{remark-Corollary-implies-Theorem-V=W}
 The condition $(ii)$  (resp.  $(iii)$, $(iv)$, $(v)$ with the changes $2)$)  of Corollary \ref{criterium-dense-vector-space-unique-equilibrium}   with $V$ $\sigma$-compact is  equivalent  to the  condition $(iv)$
 (resp. $(viii)$, $(ix)$, $(vii)$  with the changes $b)3)$) 
 of Theorem \ref{comparison-theorem} with  $\Sigma=W=V=E$ and replacing  $V\setminus\{0\}$ by $V$. Therefore (by  the properties $a)1)$ and $c)$ of Theorem \ref{comparison-theorem})  given $V$ as in Corollary \ref{criterium-dense-vector-space-unique-equilibrium}, 
 all the conditions   in part $a)$  of Theorem \ref{comparison-theorem}   (as well as those obtained with the changes $b)1)$, $b)2)$, $b)3)$) hold with $W$ any $\sigma$-compact vector space dense in $V$ and  replacing  
 $V\setminus\{0\}$ by $V$ in $(iv)$,  $(vi)$, $(vii)$, $(viii)$, $(ix)$. 
     \end{remark}

It is worth putting  into perspective the achievements of Theorem \ref{comparison-theorem} as regards  their proofs and respective relevance: The equivalences of Theorem \ref{comparison-theorem} involving (D) (except (D)$\Leftrightarrow(i)$,  which is only technical and given by  Lemma \ref{lemma-conv-net})  can be classified in three  categories
  depending on the terms in  which (D) is formulated: 
  \begin{nitemize}
  \item[-]  variational  form of the pressure in an infinite dimensional setting  in  $(iii)$, $(iv)$;
    \item[-]    topological   form of the pressure in $(v)$, $(vi)$, $(vii)$, $b)1)$, $b)2)$, $b)3)$ (\ie involving the large deviations);
    \item[-]  variational  form of the pressure  or entropy in a finite dimensional setting in $(viii)$,  $(ix)$.
  \end{nitemize}
   Once the equivalences of the first above class are obtained, the ones of the second class are  essentially based  on  the variational principle (thanks to Lemma \ref{existence-net-generating the pressure}); deriving  those of  the last  class from the preceding ones 
   is done  by means of   convex analysis,  in finite as well as  infinite dimensional setting.    So, regardless the large deviation aspect,  the main achievement of Theorem \ref{comparison-theorem}  is certainly the equivalences (D)$\Leftrightarrow(iii)\Leftrightarrow(iv)\Leftrightarrow(viii)\Leftrightarrow(ix)$ and especially 
   (D)$\Leftrightarrow(iv)$, the proof of which  relies on a variant of two results  of \cite{Israel-Phelps-84-MAGHSCAN-54} (namely, Theorem \ref{Israel-Phelps-theorem}
 and Theorem \ref{Israel-Phelps-theorem-converse-affine} of  Appendix A)  given by Lemma \ref{Israel-Phelps-theorem-strenght} and Lemma \ref{Israel-Phelps-theorem-converse-affine-strenght}; Schauder basis are   used for the proof of the implication (D)$\Rightarrow(iv)$.

 Form the large deviation point of view,  the fact that (D) implies the level-2 large deviation results $(v)$  with the changes b)1), $(vi)$  with the changes  b)2) and the form of the rate function given in a)2) is  well-known (\cite{Comman(2009)NON22}, Theorem 5.2); the statements $(v)$ and $(vi)$  follow also from Theorem 5.2 of  \cite{Comman(2009)NON22} once it is known that  $(\nu^\tau_{f,\alpha}, t^\tau_\alpha)$ fulfils  (\ref{introduction-eq30}), which is given  by Lemma \ref{existence-net-generating the pressure}. The novelties   are  the converse implications, which  follow easily from (D)$\Leftrightarrow(iv)$
   and  the form of the level-2 rate function (Lemma \ref{Fenchel-Legendre-transform-of-Q_f}).

 The key ingredient that allows to prove the equivalence between  (D)
  and  the level-1 large deviations results $(vii)$ (resp. $(vii)$ with the changes   b)3))  and especially  the essential strict convexity of the rate function   is  the peculiar form of the limiting log-moment generating function (\cf Lemma \ref{uniqueness-equil-implies-diffe-logmoment}
 and Lemma \ref{diffe-logmoment-implies-uniqueness-equil});   since the proof does not depend on the nature of the nets satisfying these  large deviation principles,  once  the existence of  such a net is established (which is given by $(vii)$) one also  obtains      (D)$\Leftrightarrow(viii)$; once  the latter is known, the equivalence     (D)$\Leftrightarrow(ix)$ relies   on  Lemma \ref{convexity-If} and Lemma \ref{equality-IS-LS*}.

\begin{example}\label{ex-TCE} 
Let  $(\Omega,\tau)$ be the  system given by the iteration of a rational map $T$ of degree at least two (\cite{Beardon-91}). More precisely, 
$\Omega$ is the Julia set of $T$ endowed with the induced chordal metric, and  the action   $\tau$ is defined by
\[\N\cup\{0\}\ni n\mapsto\tau(n) = (T_{\mid\Omega})^{n};\]
such a system  has a unique measure of maximal entropy (\cite{Ljubich-83-ETDS-3}).
We assume furthermore that $T$ fulfils a weak  form of hyperbolicity,  the so-called  ÒTopological Collet-EckmanÓ (TCE) condition: There exists $\lambda> 1$ such that every periodic point $p\in\Omega$ with period $n$ satisfies
\[
\mid (T^n)'(p)\mid\ge\lambda^n.
\]
(see Main Theorem of \cite{Przytycki_Rivera-Letelier_Smirnov(2003)InventMath151}
 for other equivalent definitions).   Let $H$ denote the space of H\"{o}lder continuous functions on $\Omega$; 
 since  every  element of $H$ has a 
 unique equilibrium state (Theorem A of   \cite{Comman_Rivera-Letelier(2010)ETDS31}) we obtain the following: 
 
 \begin{nitemize}
\item   
  The conditions $(ii)$, $(iii)$, $(iv)$, $(v)$  of  Corollary \ref{criterium-dense-vector-space-unique-equilibrium}
     hold  with $V=H$.  All the  conditions in part $a)$ (and thus also  those given by the changes $b)1)$, $b)2)$, $b)3)$) of Theorem \ref{comparison-theorem} hold; consequently, the hypotheses  (and thus the properties $a)$, $b)$, $c)$) of Theorem \ref{use-Schauder-basis-to-get-V} are fulfilled. 
Furthermore, when $W$ is a $\sigma$-compact vector space dense in $H$,  the conditions      $(iv)$,  $(vi)$, $(vii)$, $(viii)$, $(ix)$ of Theorem \ref{comparison-theorem} hold with $V=H$ and  replacing   $V\setminus\{0\}$ by $V$  (\cf Remark \ref{remark-Corollary-implies-Theorem-V=W}).

 \item   The foregoing  allows us  to generalize and strengthen  all  the level-2 as well as level-1 large deviation  results of \cite{Comman_Rivera-Letelier(2010)ETDS31};  as regards  the level-2,   the strengthening  consists  in the property of the rate function given by $(v)$ and $b)1)$; with respect to level-1, the improvement is given by    the essential strict convexity of the rate function  and  the fact that this holds   in any finite dimension; in both cases, the generalization is done by considering other nets (for instance,  
 $(\nu^\tau_{f,\alpha}, t^\tau_\alpha)$) as well as allowing  the  parameter function $f$ (as in (\ref{introduction-eq30})) to be an arbitrary element of $C(\Omega)$. More specifically, the connection with  \cite{Comman_Rivera-Letelier(2010)ETDS31}  works as follows:

   \begin{itemize}
   \item[-] The level-2  large deviation results of \cite{Comman_Rivera-Letelier(2010)ETDS31} (namely, Theorem B)  follow from   Theorem A of   \cite{Comman_Rivera-Letelier(2010)ETDS31} together with 
    Theorem 5.2 of  \cite{Comman(2009)NON22}  in place of Theorem C of \cite{Comman_Rivera-Letelier(2010)ETDS31};
  alternatively, in place of  Theorem 5.2 of  \cite{Comman(2009)NON22}  one can use the condition $(vi)$ and $b)2)$ of Theorem \ref{comparison-theorem} with $\Sigma=\{0\}$ (the equality in $b)2)$ for the nets considered in  \cite{Comman_Rivera-Letelier(2010)ETDS31} is proved in Lemmas 4.2, 4.3 and 4.4 of that paper).  Furthermore, the statements $(v)$,  $(vi)$ and $(vii)$ of Theorem \ref{comparison-theorem} furnish  large deviations  results that have not been considered in \cite{Comman_Rivera-Letelier(2010)ETDS31}.   It is easy to derive them using  Theorem A  and Theorem C of \cite{Comman_Rivera-Letelier(2010)ETDS31} and Lemma \ref{existence-net-generating the pressure}
 when $f_0\in H$ and $\Sigma\subset H$; 
  however, the statement $(v)$ when $f_0\not\in H$  and the statements $(vi)$, $(vii)$ when  $\Sigma\not\subset H$ required (D) and thus cannot be obtained directly  from \cite{Comman_Rivera-Letelier(2010)ETDS31}; indeed,  one need to know that Theorem A  of \cite{Comman_Rivera-Letelier(2010)ETDS31} implies (D) (in other words, $(ii)\Rightarrow(i)$ of  Corollary \ref{criterium-dense-vector-space-unique-equilibrium}).
      
  \item[-] The level-1  large deviation results of \cite{Comman_Rivera-Letelier(2010)ETDS31} (namely, Corollary 1.1 of that paper) are strengthened by the statements  $(vii)$ together with $b)3)$ of Theorem \ref{comparison-theorem}  (alternatively, by $(v)$ and the assertion  $2)$ of Corollary \ref{criterium-dense-vector-space-unique-equilibrium}) by  giving a $n$-dimensional version for all $n\in\N$  and 
 establishing the essential strict convexity of the rate function; in particular, the last assertion of Corollary 1.1 of \cite{Comman_Rivera-Letelier(2010)ETDS31} is a direct consequence of this fact.  The same improvements apply to Theorem 3.5  of \cite{Comman(2009)NON22}  concerning the level-1  large deviation result for hyperbolic rational maps (\ie expanding on $\Omega$): indeed, the  hyperbolicity implies that both the parameter function (\ie $-t\log\mid T'\mid$) as well as the function by which the net is pushed forward   (\ie $\log\mid T'\mid$) belong to $H$. 
 \end{itemize}
    \end{nitemize}
\end{example}

\begin{remark}\label{remark-recovering-Theorem C-of-Comman_Rivera-Letelier(2010)ETDS31}
The implication  $(iv)\Rightarrow(v)$  with the   changes  $b)1)$ of  Theorem \ref{comparison-theorem} strengthens the large deviation results of Theorem C of \cite{Comman_Rivera-Letelier(2010)ETDS31} (as well as  their   generalization  given by  Remark  B.2)   by weakening the general hypothesis, \ie allowing $f_0$ to have several  equilibrium states, and keeping the conclusions (namely, the first and the last assertions)  unchanged. 
\end{remark}

\section{Proofs}\label{Proofs}

\begin{lemma}\label{lemma-conv-net}
Let $J$ and $L$ be  directed sets,  let $s$ be a real-valued function on $J\times L$, let  $\wp$ denote the set $J\times L^{J}$ pointwise directed,  let 
  $(s_i)_{i\in\wp}$ be the  net in $\R$ defined by putting $s_i=s(j,u(j))$ for all  $i=(j,u)\in\wp$. For each  $r\in\R$ we have
\[\limsup_i\ s_i\le
r\Longleftrightarrow\limsup_j\limsup_l\
s({j,l})\le r;\]
in particular,
\[\lim_i\ s_i=r\Longleftrightarrow\liminf_j\liminf_l\
s({j,l})=\limsup_j\limsup_l\
s({j,l})= r.\]
\end{lemma}

\begin{proof}
Let $\delta>0$. First assume that $\limsup_i\ s_i\le r$. There
exists $j_0\in J$ and $u_0\in L^J$
such that $s({j,u(j)})<r+\delta$  for all
$(j,u)$ greater than or equal  $(j_0,u_0)$. Suppose that
$\limsup_j\limsup_l\ s({j,l})>r+\delta$. There exists   $(j_1,l_1)$ in $J\times L$ with $j_1$ (resp. $l_1$)   greater than or equal $j_0$ (resp. $u_0(j_1)$)   such that     $s({j_1, l_1})>r+\delta$. Putting
  $u_1(j_1)=l_1$ and $u_1(j)=u_0(j)$   for all $j\in J\setminus\{l_1\}$,  we get  an element $(j_1,u_1)\in\wp$ greater than or equal $(j_0,u_0)$ fulfilling
$s({j_1,u_1(j_1)})>r+\delta$, which gives  the
contradiction. Therefore, we have  $\limsup_j\limsup_l\
s({j,l})\le r+\delta$ hence $\limsup_j\limsup_l\
s({j,l})\le r$  since $\delta$ is
arbitrary. 

Assume now that  $\limsup_j\limsup_l\
s({j,l})\le r$.  There exists $j_0\in J$  and for each $j\in J$ greater than or equal $j_0$ there exists 
  $u_0(j)\in L$ such that  
$s({j,l})<r+\delta$ for all $j$ and $l$  greater than or equal $j_0$ and $u_0(j)$, respectively.
Putting  $u_0(j)=u_0(j_0)$  for all  $j$ lesser than  $j_0$, we get an element  $(j_0,u_0)\in\wp$ such that 
$s({j,u(j)})<r+\delta$ for all $(j,u)\in\wp$ greater than or equal $(j_0,u_0)$; therefore,  $\limsup_i s_i\le r+\delta$ hence 
$\limsup_i s_i\le r$  since $\delta$ is
arbitrary. The first assertion is proved; the second assertion is a direct consequence since 
 $\liminf_i\ s_i\ge r$ if and only if $-\limsup_i\ -s_i\ge r$  if and only if $\limsup_i\ -s_i\le -r$ if and only if $\limsup_j\limsup_l\
-s({j,l})\le -r$ if and only if $-\liminf_j\liminf_l\
s({j,l})\le -r$ if and only if $\liminf_j\liminf_l\
s({j,l})\ge r$ (where the third equivalence follows from the  first assertion applied to the net $(-s_i)$ and the real  $-r$).
\end{proof}

\begin{lemma}\label{existence-net-generating the pressure}
For each $(f,g)\in C(\Omega)^2$ we have 
\[\lim_\alpha t_\alpha\log\int_{\mathcal{M}(\Omega)}e^{({t_\alpha^{\tau}})^{-1}\int_\Omega g(\omega)\mu(d\omega)}\nu^{\tau}_{f,\alpha}(d\mu)=P^\tau(f+g)-P^{\tau}(f).\]
\end{lemma}

\begin{proof}
Let $(f,g)\in C(\Omega)^2$.
For each $(a,h,\xi)\in \N^l\times C(\Omega)\times\Omega$ we 
 put 
\[S^{\tau}_{\Lambda(a)}(h)(\xi)=\sum_{x\in\Lambda(a)}h(\tau^x\xi).\] 
We have
\begin{multline*}
P^{\tau}(f+g)-P^{\tau}(f)=
\\
\lim_{\varepsilon\rightarrow 0}\liminf_a\frac{1}{|\Lambda(a)|}\log\sum_{\xi\in
\Omega_{\varepsilon,a}}e^{S^{\tau}_{\Lambda(a)}(f+g)(\xi)}
- \lim_{\varepsilon\rightarrow 0}\limsup_a\frac{1}{|\Lambda(a)|}\log\sum_{\xi\in
\Omega_{\varepsilon,a}}e^{S^{\tau}_{\Lambda(a)}(f)(\xi)}
\\
\le\lim_{\varepsilon\rightarrow 0}\liminf_a\frac{1}{|\Lambda(a)|}\log\left(\frac{\sum_{\xi\in
\Omega_{\varepsilon,a}}e^{S^{\tau}_{\Lambda(a)}(f+g)(\xi)}}{\sum_{\xi\in
\Omega_{\varepsilon,a}}e^{S^{\tau}_{\Lambda(a)}(f)(\xi)}}\right)
\\
\le
\lim_{\varepsilon\rightarrow 0}\limsup_a\frac{1}{|\Lambda(a)|}\log\left(\frac{\sum_{\xi\in
\Omega_{\varepsilon,a}}e^{S^{\tau}_{\Lambda(a)}(f+g)(\xi)}}{\sum_{\xi\in
\Omega_{\varepsilon,a}}e^{S^{\tau}_{\Lambda(a)}(f)(\xi)}}\right)
\\
\le\lim_{\varepsilon\rightarrow 0}\limsup_a\frac{1}{|\Lambda(a)|}\log\sum_{\xi\in
\Omega_{\varepsilon,a}}e^{S^{\tau}_{\Lambda(a)}(f+g)(\xi)}
\\
-\lim_{\varepsilon\rightarrow 0}
\liminf_a\frac{1}{|\Lambda(a)|}\log\sum_{\xi\in
\Omega_{\varepsilon,a}}e^{S^{\tau}_{\Lambda(a)}(f)(\xi)}=P^{\tau}(f+g)-P^{\tau}(f),
\end{multline*}
hence
\begin{multline}\label{existence-net-generating the pressure-eq40}
\lim_{\varepsilon\rightarrow 0}\liminf_a\frac{1}{|\Lambda(a)|}\log\left(\frac{\sum_{\xi\in
\Omega_{\varepsilon,a}}e^{S^{\tau}_{\Lambda(a)}(f+g)(\xi)}}{\sum_{\xi\in
\Omega_{\varepsilon,a}}e^{S^{\tau}_{\Lambda(a)}(f)(\xi)}}\right)
\\
=\lim_{\varepsilon\rightarrow 0}\limsup_a\frac{1}{|\Lambda(a)|}\log\left(\frac{\sum_{\xi\in
\Omega_{\varepsilon,a}}e^{S^{\tau}_{\Lambda(a)}(f+g)(\xi)}}{\sum_{\xi\in
\Omega_{\varepsilon,a}}e^{S^{\tau}_{\Lambda(a)}(f)(\xi)}}\right).
\end{multline}
 Since for each
 $\alpha=(\varepsilon,u)\in\wp$,
\[t^\tau_\alpha\log\int_{\mathcal{M}(\Omega)}
e^{(t^\tau_\alpha)^{-1}\int_\Omega g(\omega)\mu(d\omega)}\nu^\tau_{f,\alpha}(d\mu)=\frac{1}{|\Lambda (u(\varepsilon))|}\log\left(\frac{\sum_{\xi\in
\Omega_{\varepsilon,u(\varepsilon)}}e^{S^{\tau}_{u(\varepsilon)}(f+g)(\xi)}}{\sum_{\xi\in
\Omega_{\varepsilon,u(\varepsilon)}}e^{S^{\tau}_{u(\varepsilon)}(f)(\xi)}}\right),\]
the conclusion follows from (\ref{existence-net-generating the pressure-eq40}) together with 
  Lemma \ref{lemma-conv-net} applied with $J=\ ]0,+\infty[$ and  $L=\N^l$ with their respective orders (defined before Theorem \ref{comparison-theorem}) and $s$ defined by 
  \[\forall (\varepsilon,a)\in J\times L,\ \ \ \ \ \ \ \ s(\varepsilon,a)=\frac{1}{|\Lambda (a)|}\log\left(\frac{\sum_{\xi\in
\Omega_{\varepsilon,a}}e^{S^{\tau}_{a}(f+g)(\xi)}}{\sum_{\xi\in
\Omega_{\varepsilon,u(\varepsilon)}}e^{S^{\tau}_{a}(f)(\xi)}}\right).\]
     \end{proof}

\begin{lemma}\textnormal{\textbf{(Phelps)}}\label{ergodic-equal-unique-equilibrium}
For each $\mu\in\mathcal{M}^{\tau}(\Omega)$, 
$\mu\in\sE^\tau(\Omega)$ if and only if $\mu$ is the unique equilibrium state for some element in $C(\Omega)$.
\end{lemma}

\begin{proof}
The fact that each  $\mu\in\sE^\tau(\Omega)$  is the unique equilibrium state for some element in $C(\Omega)$ is exactly 
Theorem 1 of \cite{Phelps_Dynamics and Randomness(2002)Santiago}; the converse follows from the fact that $h^\tau$ is affine. 
\end{proof}

Recall that   the directional derivative of  $P^{\tau}$ at $f$ in the direction $g$ is denoted by   $dP^{\tau}(f;g)$  for all  $(f,g)\in C(\Omega)^2$  (\cf \S \ref{subsection-Convex analysis}).

\begin{lemma}\label{Gateaux-diffe-in-the-direction-dense-subspace}
Let $f\in C(\Omega)$ and let $W$ be a vector space dense in $C(\Omega)$.  If the map $W\ni g\mapsto dP^{\tau}(f;g)$ is real-valued and linear,  then $P^{\tau}$ is Gateaux differentiable at $f$. 
\end{lemma}

\begin{proof}
Let $\delta P^{\tau}(f)$ denote the set of subgradients of $P^\tau$ at $f$ (\ie $\delta P^{\tau}(f)$ is the set of equilibrium states for $f$,  \cf \S \ref{subsection-Convex analysis}).  
   For each   $(\mu,g)\in\delta P^{\tau}(f)\times W$ 
  we have
  \[\forall\varepsilon>0,\ \ \ \ \ \ \ P^{\tau}(f+\varepsilon g)-P^{\tau}(f)\ge\varepsilon\mu(g);\]
  dividing by $\varepsilon$ and letting $\varepsilon\rightarrow 0$ yields
  $dP^{\tau}(f;g)-\mu(g)\ge 0$ 
  hence $dP^{\tau}(f;g)=\mu(g)$ by linearity of $dP^{\tau}(f;\cdot)_{\mid W}$ and $\mu_{\mid W}$.  Since $W$ is 
 dense in $C(\Omega)$  it follows that  $\delta P^{\tau}(f)$ is a singleton, which proves the lemma (\cite{Ekeland_Teman},  Proposition 5.3;  \cf \S \ref{subsection-Convex analysis}).
   \end{proof}

For each  $((f,f_1...,f_n),n)\in C(\Omega)^{n+1}\times\N$ let   $L_{f,(f_1,...,f_n)}$  be the function defined on $\R^n$ by    
   \[
 \forall (t_1,...,t_n)\in\R^n,\ \ \ \ \ \ \ \ \ \ L_{f,(f_1,...,f_n)} ((t_1,...,t_n))=P^{\tau}\left(f+\sum_{k=1}^n t_k f_k\right)-P^{\tau}(f); 
  \]
  clearly,  $L_{f,(f_1,...,f_n)}$ is real-valued and convex; note that when    $(\nu_\alpha,t_\alpha)$ is a net as in   (\ref{introduction-eq30}),  $L_{f,(f_1,...,f_n)}$ is  the limiting log-moment generating function associated with the net $((\widehat{f_1},...,\widehat{f_n})[\nu_{\alpha}],t_\alpha)$ (\S \ref{subsection-Large deviations}).

\begin{lemma}\label{uniqueness-equil-implies-diffe-logmoment}
Let   $((f,f_1,...,f_n),(t_1,...,t_n),n)\in C(\Omega)^{n+1}\times\R^n\times\N$.  If  
\[\forall j\in\{1,...,n\},\ \ \ \ \ \ \ \ dP^{\tau}(f+\sum_{k=1}^n t_k f_k;f_j)=-dP^{\tau}(f+\sum_{k=1}^n t_k f_k;-f_j),\]   then 
  $L_{f,(f_1,...,f_n)}$ is differentiable at $(t_1,...,t_n)$.
  \end{lemma}

\begin{proof}
For each $\varepsilon>0$ and for each $j\in\{1,...,n\}$ we have 
\begin{multline*}
L_{f,(f_1,...,f_n)}((t_1,...,t_j+\varepsilon,..,t_n))-L_{f,(f_1,...,f_n)}((t_1,..,t_n))
\\
=P^{\tau}\left(f+\sum_{k=1}^n t_k f_k+\varepsilon f_j\right)-P^{\tau}\left(f+\sum_{k=1}^n t_k f_k\right)
\end{multline*}
and
\begin{multline*}
L_{f,(f_1,...,f_n)}((t_1,...,t_j-\varepsilon,..,t_n))-L_{f,(f_1,...,f_n)}((t_1,..,t_n))
\\
=P^{\tau}\left(f+\sum_{k=1}^n t_k f_k-\varepsilon f_j\right)-P^{\tau}\left(f+\sum_{k=1}^n t_k f_k\right).
\end{multline*}
Dividing by $\varepsilon$ and  letting $\varepsilon\rightarrow 0$ in   both above expressions   shows that  the $j^{th}$-partial derivative  of $L_{f,(f_1,...,f_n)}$ at $(t_1,...,t_n)$ exists and fulfils
 \[\frac{\delta^j L_{f,(f_1,...,f_n)}}{\delta t_j}(t_1,...,t_n)=dP^{\tau}(f+\sum_{k=1}^n t_k f_k;f_j); 
 \]
 furthermore, $\frac{\delta^j L_{f,(f_1,...,f_n)}}{\delta t_j}(t_1,...,t_n)$ is finite since  the effective domain of $L_{f,(f_1,...,f_n)}$ is  $\R^n$;   the conclusion follows from Theorem 25.2 of  \cite{Rockafellar-70}.
\end{proof}

\begin{lemma}\label{diffe-logmoment-vs-uniqueness-equil}
Let $f\in C(\Omega)$   and  let $W$ be a vector space dense in $C(\Omega)$. If  $L_{f,(g,h)}$  is differentiable at   zero for all 
$(g,h)\in W^2$, then $P^{\tau}$ is Gateaux differentiable at $f$.  
\end{lemma}

\begin{proof}
 For each  $((g,h),(\lambda,\lambda'))\in W^2\times\R^2$ and for each $\varepsilon>0$ we have
 \[
L_{f,(g,h)}(\varepsilon (\lambda,\lambda'))=P^{\tau}(f+\varepsilon(\lambda g+\lambda' h))-P^\tau(f); 
\]
dividing by $\varepsilon$ and letting $\varepsilon\rightarrow 0$  yields
 \[
  dP^{\tau}(f;\lambda g+\lambda' h)=D_{(0,0)} L_{f,(g,h)}((\lambda,\lambda'))
\]
hence  $dP^{\tau}(f;\cdot)_{\mid W}$ is real-valued and linear on $W$; the conclusion follows from Lemma \ref{Gateaux-diffe-in-the-direction-dense-subspace}.
\end{proof}

\begin{lemma}\label{diffe-logmoment-implies-uniqueness-equil}
Let   $f\in C(\Omega)$ and  let $E$ be a set generating a vector space  $W$  dense in $C(\Omega)$.  If  $L_{f,(f_1,...,f_n)}$ is differentiable on  $\R^n$  for  all  $((f_1,...,f_n),n)\in E^n\times\N$, then $P^{\tau}$ is Gateaux differentiable at  $f+g$ for all $g\in W$.
  \end{lemma}

\begin{proof}
For each $((f_1,...,f_{\max\{m,m_1,m_2\}}),(t_1,...,t_m,u_1,...,u_{m_1}, v_1,...,v_{m_2}),(m,m_1,m_2))\in E^{\max\{m,m_1,m_2\}}\times\R^{m+m_1+m_2}\times\N^3$,  for each $\varepsilon>0$ and for each $\varepsilon'>0$ we have 
 \begin{multline*}
 L_{f+\sum_{k=1}^m t_k f_k, (\sum_{k=1}^{m_1} u_k f_k, \sum_{k=1}^{m_2} v_k f_k)} ((\varepsilon,\varepsilon'))
 \\
 = L_{f+\sum_{k=1}^m t_k f_k, (f_1,...,f_{m_1}, f_1,...,f_{m_2})}((\varepsilon  u_1,...,\varepsilon u_{m_1}, \varepsilon' v_1,...,\varepsilon' v_{m_2}))
 \\
 = (P^{\tau}(f+\sum_{k=1}^m t_k f_k+\varepsilon\sum_{k=1}^{m_1} u_k f_k + \varepsilon'\sum_{k=1}^{m_2} v_k f_k) 
- P^{\tau}(f)) + (P^{\tau}(f) - P^{\tau}(f+\sum_{k=1}^m t_k f_k))
 \\
=L_{f, (f_1,...,f_m,f_1,...,f_{m_1}, f_1,...,f_{m_2})}((t_1,...,t_m,\varepsilon u_1,...,\varepsilon u_{m_1},\varepsilon' v_1,...,\varepsilon' v_{m_2}))
\\
-L_{f, (f_1,...,f_m,f_1,...,f_{m_1}, f_1,...,f_{m_2})}((t_1,...,t_m,0,...,0))
 \end{multline*}
 so that the differentiability of $L_{f, (f_1,...,f_m,f_1,...,f_{m_1}, f_1,...,f_{m_2})}$ at $(t_1,...,t_m,0,...,0)$ implies the differentiability of 
 $L_{f+\sum_{k=1}^m t_k f_k, (\sum_{k=1}^{m_1} u_k f_k, \sum_{k=1}^{m_2} v_k f_k)}$
  at $0$; the  conclusion  follow from Lemma \ref{diffe-logmoment-vs-uniqueness-equil}. 
    \end{proof}

For each $f\in C(\Omega)$ let $Q_f$ be the map defined on $C(\Omega$ by 
\[\forall g\in C(\Omega),\ \ \ \ \ \ \ \ \ Q_f(g)=P^{\tau}(f+g)-P^{\tau}(f).\]

\begin{lemma}\label{Fenchel-Legendre-transform-of-Q_f}
 For each $f\in C(\Omega)$
the function $Q_f$ is proper convex continuous and  its Legendre-Fenchel transform $Q_f^*$  fulfils 
\[
\forall\mu\in\widetilde{\mathcal{M}}(\Omega),\ \ \ \ \ \ \ \  Q_f^*(\mu)=
\left\{
\begin{array}{ll}
P^{\tau}(f)-h^{\tau}(\mu)-\mu(f) & \textnormal{if $\mu\in\mathcal{M}^{\tau}(\Omega)$}
\\ 
+\infty & \textnormal{if $\mu\in\widetilde{\mathcal{M}}(\Omega)\setminus \mathcal{M}^{\tau}(\Omega)$}.
\end{array}
\right.
\]
In particular,  $Q_f^*$ vanishes exactly on the set of equilibrium states for $f$. 
\end{lemma}

\begin{proof}
Let $f\in C(\Omega)$.   Clearly, $Q_f$ is proper  convex and continuous  since $P^{\tau}$ and $\widehat{f}$ are (\cf \S \ref{subsection-Thermodynamic formalism} and \S \ref{subsection-Convex analysis}). 
 Putting
\[
U(\mu)=\left\{
\begin{array}{ll}
-\mu(f)-h^{\tau}(\mu) & \textnormal{if $\mu\in\mathcal{M}^{\tau}(\Omega)$}
\\ 
+\infty & \textnormal{if $\mu\in\widetilde{\mathcal{M}}(\Omega)\setminus \mathcal{M}^{\tau}(\Omega)$},
\end{array}
\right.
\]
 we have
\[P^{\tau}(f+g)=\sup_{\mathcal{M}^{\tau}(\Omega)}\{\mu(f+g)+h^{\tau}(\mu)\}=
\sup_{\mu\in\widetilde{\mathcal{M}}(\Omega)}
\{\mu(g)-U(\omega)\}.\]
Since $h^{\tau}$ is bounded  affine and upper semi-continuous  (\cf \S \ref{subsection-Thermodynamic formalism}),  $U$ is proper convex and  lower semi-continuous; consequently, we have $U=U^{**}$ (\cf \S \ref{subsection-Convex analysis}) \ie 
\[\forall \mu\in\widetilde{\mathcal{M}}(\Omega),\ \ \ \ \ \ \ U(\mu)=\sup_{g\in C(\Omega)}\{\mu(g)-P^{\tau}(f+g)\}=
\sup_{g\in C(\Omega)}\{\mu(g)-P^{\tau}(f)-Q_f(g)\}=\]
\[-P^{\tau}(f)+\sup_{g\in C(\Omega)}\{\mu(g)-Q_f(g)\}=-P^{\tau}(f)+Q_f^*(\mu),\]which
proves the lemma. 
 \end{proof}

For each $((f,f_1,...,f_n),n)\in C(\Omega)^{n+1}\times\N$ let $I_{f,(f_1,...,f_n)}$  be the function defined on $\R^n$ by 
\[\forall x\in\R^n,\ \ \ \ \ \ \ \ I_{f,(f_1,...,f_n)}(x)=\inf\{Q_f^*(\mu):\mu\in\mathcal{M}(\Omega), (\mu(f_1),...,\mu(f_n))=x\}.\]
Since $\mathcal{M}^{\tau}(\Omega)$ is compact and $Q_f^*$ is lower semi-continuous, $I_{f,(f_1,...,f_n)}(x)$ is a minimum if and only if $x\in (\widehat{f_1},...,\widehat{f_n})(\mathcal{M}^{\tau}(\Omega))$  by Lemma \ref {Fenchel-Legendre-transform-of-Q_f}, hence
\begin{equation}\label{If-propre}
I_{f,(f_1,...,f_n)}(x)=\left\{
\begin{array}{ll}
Q_f^*(\mu_x) \textnormal{\ for some
$\mu_x\in\mathcal{M}^{\tau}(\Omega)$} & \textnormal{if $x\in (\widehat{f_1},...,\widehat{f_n})(\mathcal{M}^{\tau}(\Omega))$}
\\
+\infty &  \textnormal{if $x\not\in (\widehat{f_1},...,\widehat{f_n})(\mathcal{M}^{\tau}(\Omega))$}.
\end{array}
\right.
\end{equation}

\begin{lemma}\label{convexity-If}
The function  $I_{f,(f_1,...,f_n)}$  is proper convex and  lower semi-continuous for all $((f,f_1,...,f_n),n)\in C(\Omega)^{n+1}\times\N$.
\end{lemma}

\begin{proof}
Let $((f,f_1,...,f_n),n)\in C(\Omega)^{n+1}\times\N$, 
let $x\in\mathbb{R}^n$, let $(x_i)$ be a net in $\R^n$ converging to $x$, and
assume that $\liminf I_{f,(f_1,...,f_n)}(x_i)<\delta$ for some real $\delta$. There exits a 
subnet $(x_j)$ of $(x_i)$ such that eventually $I_{f,(f_1,...,f_n)}(x_j)<\delta$ and thus
$Q_f^*(\mu_j)<\delta$ for some $\mu_j\in\mathcal{M}^{\tau}(\Omega)$
satisfying $(\mu_j(f_1),...,\mu_j(f_n))=x_j$. Let $(\mu'_j)$ be a subnet of $(\mu_j)$
converging to some $\mu'\in\mathcal{M}^{\tau}(\Omega)$; note
that $(\mu'(f_1),...,\mu'(f_n))=x$. We get
\[I_{f,(f_1,...,f_n)}(x)\le Q_f^*(\mu')\le\liminf\  Q_f^*(\mu'_j)<\delta,\]
which proves the lower semi-continuity of $I_{f,f_1,...,f_n}$. For each $(x_1,x_2,\beta)\in\R^n\times\R^n\times\ ]0,1[$ we have
\begin{multline*}
I_{f,(f_1,...,f_n)}(\beta x_1+(1-\beta)x_2)=\inf\{Q_f^*(\mu):
\mu\in\mathcal{M}(\mu),(\mu(f_1),...,\mu(f_n))=\beta
x_1+(1-\beta)x_2\}
\\
\le\inf\{Q_f^*(\beta\mu_1+(1-\beta)\mu_2):(\mu_1,
\mu_2)\in\mathcal{M}(\Omega)^2,(\mu_1(f_1),...,\mu_n(f_n))=x_1,
\\
(\mu_2(f_1),...,\mu_2(f_n))=x_2\}
\\
\le\inf\{\beta Q_f^*(\mu_1)+(1-\beta)Q_f^*(\mu_2):(\mu_1,\mu_2)\in\mathcal{M}(\Omega)^2,(\mu_1(f_1),...,\mu_m(f_n))=x_1,
\\
(\mu_2(f_1),...,\mu_2(f_n))=x_2\}\le\beta I_{f,(f_1,...,f_n)}(x_1)+(1-\beta)I_{f,(f_1,...,f_n)}(x_2),
\end{multline*}
 hence $I_{f,f_1,...,f_n}$ is convex; $I_{f,(f_1,...,f_n)}$ is proper  by (\ref{If-propre}).
 \end{proof}

\begin{lemma}\label{equality-IS-LS*}
 We have $I_{f,(f_1,...,f_n)}=L_{f,(f_1,...,f_n)}^*$ for all $((f,f_1,...,f_n),n)\in C(\Omega)^{n+1}\times\N$.
 \end{lemma}

\begin{proof}
Let  $((f,f_1,...,f_n),n)\in C(\Omega)^{n+1}\times\N$ and let $\langle\ ,\ \rangle$ denote the scalar product in $\R^n$.
Suppose  there exists  $t=(t_1,...,t_n)\in\mathbb{R}^n$ such that 
\[\sup_{x\in\mathbb{R}^n}\{\langle
t,x\rangle-I_{f,(f_1,...,f_n)}(x)\}<L_{f,(f_1,...,f_n)}(t).\]  Since
\[L_{f,(f_1,...,f_n)}(t)=Q_f\left(\sum_{k=1}^n t_k f_k\right)=Q_f^{**}\left(\sum_{k=1}^n t_k f_k\right)=\sup_{\mu\in\mathcal{M}^{\tau}(\Omega)}\{\langle t,(\mu(f_1),...,\mu(f_n))\rangle -Q_f^*(\mu)\}\]
by Lemma \ref{Fenchel-Legendre-transform-of-Q_f}, 
there exists $\mu\in\mathcal{M}^{\tau}(\Omega)$ such that
\[\sup_{x\in\mathbb{R}^n}\{\langle t,x\rangle -I_{f,(f_1,...,f_n)}(x)\}<\langle
t,(\mu(f_1),...,\mu(f_n))\rangle-Q_f^*(\mu),\] which gives the contradiction by taking
$x=(\mu(f_1),...,\mu(f_n))$ in the left hand side; therefore, we have  
\begin{equation}\label{equality-IS*-LS-eq20}
\forall t\in\R^n,\ \ \ \ \ \ \ \ I_{f,(f_1,...,f_n)}^*(t)\ge L_{f,(f_1,...,f_n)}(t).
\end{equation}
 If
$\sup_{x\in\mathbb{R}^n}\{\langle t,x\rangle -I_{f,(f_1,...,f_n)}(x)\}>L_{f,(f_1,...,f_n)}(t)$ for
some  $t\in\mathbb{R}^n$, then \[\langle
t,x\rangle-I_{f,(f_1,...,f_n)}(x)>\sup_{\mu\in\mathcal{M}^{\tau}(\Omega)}\{\langle
t,(\mu(f_1),...,\mu(f_n))\rangle-Q_f^*(\mu)\}\] for some $x=(\mu_x(f_1),...,\mu_x(f_n))$ with
$\mu_x\in\mathcal{M}^{\tau}(\Omega)$ fulfilling $I_{f,(f_1,...,f_n)}(x)=Q_f^*(\mu_x)$, which gives the contradiction; consequently, we have  $I_{f,(f_1,...,f_n)}^*(t)\le L_{f,(f_1,...,f_n)}(t)$ for all $t\in\R^n$, which together with (\ref{equality-IS*-LS-eq20}) yields $I_{f,(f_1,...,f_n)}^*= L_{f,(f_1,...,f_n)}$; the conclusion follows from the foregoing equality together with Lemma \ref{convexity-If}.
\end{proof}

\begin{lemma}\label{existence-of-non-GD-element}
$C(\Omega)$ contains an element  admitting several equilibrium sates.
\end{lemma}

\begin{proof}
Since $\mathcal{M}^{\tau}(\Omega)$ is not a singleton,  Theorem 3.4 of \cite{Israel-Phelps-84-MAGHSCAN-54} ensures the existence of some $a\in A(\mathcal{M}^{\tau}(\Omega))$ where $P_{-h^\tau}$ is not Gateaux differentiable. Since the map $C(\Omega)\ni f\mapsto \widehat{f}\in A(\mathcal{M}^{\tau}(\Omega))$ is surjective,  there exists $f\in C(\Omega)$ such that $\widehat{f}=a$. By the equations  (\ref{remark-proof-Phelps-Israel-result-eq60}) in Appendix A   the function  $P^\tau$ is not Gateaux differentiable at $f$; equivalently, $f$ has several equilibrium sates (\cf \S \ref{subsection-Convex analysis}).
\end{proof}

 \begin{lemma}\label{Israel-Phelps-theorem-strenght}
   Let $W$ be a  $\sigma$-compact vector space dense in $C(\Omega)$. Assume that property $(i)$ of Theorem  \ref{comparison-theorem} holds. There exists a   $\sigma$-compact infinite dimensional 
       vector space $\widetilde{W}$ linearly independent from  $W$ such that $f+h$ has a  unique equilibrium state for all  $(f,h)\in W\times (\widetilde{W}\setminus\{0\})$.
   \end{lemma}

\begin{proof}
 (We use the notations and results of Appendix A.)   
 By  Lemma \ref{existence-of-non-GD-element} there exists 
  $f_0\in C(\Omega)$ admitting several equilibrium states.   
   Put $W_0=W+\textnormal{span}(\{f_0\})$. 
For each $n\in\N$ let $\sP(n)$ denote the following property: There exists a  $n$-dimensional  vector space $W_n$ linearly independent from $W_0$ such that $f+h$ has a  unique equilibrium state for all  $(f,h)\in W_0\times (W_n \setminus\{0\})$.
Put $S=\{\widehat{f}:f\in W_0\}$ so  that $S$ is a $\sigma$-compact  vector space dense in  $A(\mathcal{M}^{\tau}(\Omega))$. 
Let $G_1$ be the subset of $A(\mathcal{M}^{\tau}(\Omega))$ as in  Theorem \ref{Israel-Phelps-theorem} of Appendix A for the above choice of $S$; we have    $G_1\supset \{\widehat{f}+t a_1: f\in W_0, t\in\R\setminus\{0\}\}$ for some $a_1\in A(\mathcal{M}^{\tau}(\Omega))$. Let $h_1\in C(\Omega)$ such that  $\widehat{h_1}=a_1$.
The equations (\ref{remark-proof-Phelps-Israel-result-eq60}) in Appendix A  shows that $P_{-h^\tau}$ is Gateaux differentiable at $\widehat{f}+t a_1$ if and only if 
$P^{\tau}$ is Gateaux differentiable at $f+t h_1$ for all $(f,t)\in W_0\times\R\setminus\{0\}$.
Note that $h_1\not\in W_0$  (because $P^\tau$ is  not Gateaux differentiable at  $f_0$). 
Therefore, $\sP(1)$   holds    by putting $W_1=\textnormal{span}(\{h_1\})$.  Assume that $\sP(n)$
holds  for   $n\in\N$. 
Put $W'=W_0\oplus W_n$.  Since    $W'$ is a $\sigma$-compact  vector space dense in $C(\Omega)$ and containing $f_0$,  the conclusion  of the preceding case  with  $W'$ in place of $W_0$ and $n=1$
  ensures the existence of  some $h_{n+1}\in C(\Omega)\setminus W'$ such that 
   that  $f+t h_{n+1}$ has a unique equilibrium state for all $(f,t)\in W'\times\R\setminus\{0\}$. Therefore, 
  $\sP(n+1)$ holds by putting $W_{n+1}=W_n+\textnormal{span}(\{h_{n+1}\})$. Consequently, $\sP(n)$ holds for all  $n\in\N$ and 
   the conclusion follows  by putting $\widetilde{W}=\textnormal{span}(\bigcup_{n\in\N}W_n)$. 
  \end{proof}

 \begin{lemma}\label{Israel-Phelps-theorem-converse-affine-strenght}
If  property $(i)$ of Theorem  \ref{comparison-theorem} does not holds, then there exists an integer  $m\ge 2$, $\varepsilon>0$  and 
  a nonempty  open subset $G$ of   ${C(\Omega)}^m$  such that  for each 
  $(g_1,...,g_m)\in G$ and for each $g\in C(\Omega)\setminus\textnormal{span}\{g_1,...,g_m\}$  
  with $\mid\mid g\mid\mid<\varepsilon$ the function $P^{\tau}$  is not Gateaux differentiable  at  $g+\sum_{k=1}^m t_k g_k$ for some  $(t_1,..., t_m)\in\R^m$. 
 \end{lemma} 
 
\begin{proof}
(We use the notations and results of Appendix A.) 
Let $n$ be an integer as in Theorem \ref  {Israel-Phelps-theorem-converse-affine} of Appendix A and    put  $A_\varepsilon(\mathcal{M}^{\tau}(\Omega))=\{a\in A(\mathcal{M}^{\tau}(\Omega)):\mid\mid a\mid\mid<\varepsilon\}$ for all $\varepsilon>0$.
 The proof of  Theorem \ref  {Israel-Phelps-theorem-converse-affine} of Appendix A reveals that 
$n\ge 3$ and there  is a nonempty  open subset $U_{n-1}$ of   ${A(\mathcal{M}^{\tau}(\Omega))}^{n-1}$ and $\varepsilon>0$ such that   for each 
$((a_1,...,a_{n-1}),a)\in U_{n-1}\times A_\varepsilon(\mathcal{M}^{\tau}(\Omega))\setminus\textnormal{span}(\{a_1,...,a_{n-1}\})$
 the function $P_{-h^\tau}$   is not Gateaux differentiable  at  $a+\sum_{k=1}^{n-1} t_k a_k$ for some  $(t_1,...,t_{n-1})\in\R^{n-1}$. 
  Note that when $A(\mathcal{M}^{\tau}(\Omega))=\textnormal{span}(\{a_1,...,a_{n-1}\})$, the above  assertion  holds trivially by emptiness of the premise (because $U_{n-1}\times A_\varepsilon(\mathcal{M}^{\tau}(\Omega))\setminus\textnormal{span}\{a_1,...,a_{n-1}\})=\emptyset$).
     Taking account of  the equations (\ref{remark-proof-Phelps-Israel-result-eq60}) in Appendix A together with  the continuity  and linearity of the map  $C(\Omega)\ni g\mapsto \widehat{g}$, 
the conclusion follows by putting $m=n-1$ and $G=\{g\in C(\Omega):\widehat{g}\in U_{m}\}$.
 \end{proof}

\begin{lemma}\label{condition on the h_n}
Let $W$ and  $\widetilde{W}$  be  linearly independent infinite dimensional  vector subpaces of $C(\Omega)$,   let $(f_n)$ (resp. $(h_n)$) be 
  a sequence in $W$ (resp. $\widetilde{W}$) and let us consider the following inclusion:
  \begin{equation}\label{condition on the h_n-eq10}
  \textnormal{span}(\{f_n+h_n:n\in\N\})\setminus\{0\}\subset W+(\widetilde{W}\setminus\{0\}).
  \end{equation}
If the set $\{h_n:n\in\N\}$ is linearly independent, then   (\ref{condition on the h_n-eq10}) holds. If the set $\{f_n:n\in\N\}$ is linearly independent and  (\ref{condition on the h_n-eq10}) holds, then the set $\{h_n:n\in\N\}$ is linearly independent.
\end{lemma}

\begin{proof}
  Suppose  that the set $\{h_n:n\in\N\}$ is linearly independent and   (\ref{condition on the h_n-eq10}) does not hold. 
Since  $\textnormal{span}(\{f_n+h_n:n\in\N\})\subset W+\widetilde{W}$ there exists $N\in\N$,  $\{c_1,...,c_N\}\subset\R$ and $g\in W$ such that 
\begin{equation}\label{condition on the h_n-eq20}
0\neq\sum_{k=1}^N c_k(f_{n_k}+h_{n_k})=g
\end{equation}
hence 
\[\sum_{k=1}^N c_k h_{n_k}=g-\sum_{k=1}^N c_k f_{n_k}\in W.\]
Since $W\cap\widetilde{W}=\{0\}$,  the above expression implies $\sum_{k=1}^N c_k h_{n_k}=0$ hence $c_k=0$ for all $k\in\{1,...,N\}$ by  linearly independence of  $\{h_n:n\in\N\}$,  which contradicts (\ref{condition on the h_n-eq20}).

Conversely, suppose  that set $\{f_n:n\in\N\}$ is linearly independent,   (\ref{condition on the h_n-eq10}) holds and 
the set $\{h_n:n\in\N\}$  is not   linearly  independent.  Let $\{h_{n_1},...,h_{n_N}\}$ be a linearly dependent subset of $\{h_n:n\in\N\}$. There are  real numbers $c_1,...,c_N$ with 
$c_{k_0}\neq 0$ for some $k_0\in\{1,...,N\}$ such that $\sum_{k=1}^N c_k h_{n_k}=0$. The linear independence of $\{f_n:n\in\N\}$ implies
 $\sum_{k=1}^N c_k f_{n_k}\neq 0$
  hence 
\[
0\neq\sum_{k=1}^N c_k (f_{n_k}+h_{n_k})=\sum_{k=1}^N c_k f_{n_k},
\]
which together with 
   (\ref{condition on the h_n-eq10}) yields
\[
\sum_{k=1}^N c_k f_{n_k}\in W\cap(W+(\widetilde{W}\setminus\{0\})).
\]
The contradiction follows  from the above expression  since 
the hypothesis  $W\cap\widetilde{W}=\{0\}$ implies $W\cap (W+\widetilde{W}\setminus\{0\})=\emptyset$.
\end{proof}

\begin{proof-theo}\label{proof-theo}
   Let $(v_1)$  (resp. $(vi_2)$, $(vii_3)$, $(iv_4)$) denote the statement obtained from $(v)$   (resp. $(vi)$,  $(vii)$, $(iv)$) with the changes described in   $b)1)$ (resp.   $b)2)$, $b)3)$, $b)4)$).

 \medskip
 
 $\bullet$ Proof of $(iv_4)\Rightarrow(iv)$, $(v_1)\Rightarrow(v)$, $(vi_2)\Rightarrow(vi)$ and $(vii_3)\Rightarrow(vii)$, with the same space $V$ in the premise as in the conclusion regarding  the last two  implications: The first  implication is obvious; all the others as well as the clarification on $V$  follow from Lemma \ref{existence-net-generating the pressure} and the observation before Lemma \ref{uniqueness-equil-implies-diffe-logmoment}.

\medskip

 $\bullet$ Proof of $(i)\Rightarrow(ii)$:   Assume that $(i)$ holds.   Let $\mu\in\mathcal{M}^{\tau}(\Omega)$. 
For each $\varepsilon>0$ there is a sequence $(\mu_{\varepsilon,n})$ in ${\sE}^\tau(\Omega)$ converging to $\mu$ and fulfilling   $h^\tau(\mu_{\varepsilon,n})>h^\tau(\mu)-\varepsilon$ for all $n\in\N$. 
Let us consider the  product set   $\wp=\ ]0,+\infty[\times\N^{]0,+\infty[}$  directed as before Theorem \ref{comparison-theorem} (with $l=1$),    let  $(\mu_i)_{i\in\wp}$ be the net defined by putting $\mu_i=\mu_{\varepsilon,u(i)}$ for all $i=(\varepsilon,u)\in\wp$ and let 
$s$ be the function  defined on $]0,+\infty[\times\N$ by 
\[\forall (\varepsilon,n)\in\ ]0,+\infty[\times\N,\ \ \ \ \ \ \ s(\varepsilon,n)=h^\tau(\mu_{\varepsilon,n}).\]
We have 
\[\forall\varepsilon>0,\ \ \ \ \ \ \ h^\tau(\mu)\ge\limsup_n h^\tau(\mu_{\varepsilon,n})\ge\liminf_n h^\tau(\mu_{\varepsilon,n})\ge h^\tau(\mu)-\varepsilon\] and letting $\varepsilon\rightarrow 0$ yields
\[h^\tau(\mu)=\limsup_\varepsilon\limsup_n h^\tau(\mu_{\varepsilon,n})=\liminf_\varepsilon\liminf_n h^\tau(\mu_{\varepsilon,n})
\]
and
\[\lim_\varepsilon\lim_n \mu_{\varepsilon,n}=\mu.\]
The last  above expression shows that
$\lim_i \mu_i=\mu$ (\cite{Kelley-91}, Theorem on Iterated Limits, p. 69)
and the first one  together with  Lemma \ref{lemma-conv-net}  (applied with  $J=\ ]0,+\infty[$ and  $L=\N$) 
yields  $\lim_i h^\tau(\mu_i)=h^\tau(\mu)$, which proves $(ii)$.

\medskip

 $\bullet$ Proof of $(i)\Rightarrow(iv_4)$:   Assume that $(i)$ holds.  Let $\widetilde{W}$ be  as in Lemma \ref{Israel-Phelps-theorem-strenght}. Let  $(f_n)$ be a Schauder basis of $C(\Omega)$ included in $W$ (\cf \S \ref{subsection-Thermodynamic formalism}) and let   $\{h_n:n\in\N\}$ be a linearly independent subset 
 of    $\widetilde{W}$ 
fulfilling 
\[\sum_{n=1}^{+\infty}\left(\sup_{f\in C(\Omega), \mid\mid f\mid\mid\le 1}\mid \lambda_n(f)\mid\right)\mid\mid h_n\mid\mid<1\]
(with $\lambda_n(f)$ as in \S \ref{subsection-Thermodynamic formalism}).  
The sequence $(f_n+h_n)$ is a Schauder basis of $C(\Omega)$ (\cf \S \ref{subsection-Thermodynamic formalism})  hence 
$\textnormal{span}(\{f_n + h_n:n\in\N\})$ 
is dense in $C(\Omega)$; clearly, $\textnormal{span}(\{f_n + h_n:n\in\N\})$ is $\sigma$-compact; since $W\cap\widetilde{W}=\{0\}$ and $\widetilde{W}$ is infinite dimensional, $\textnormal{span}(\{f_n + h_n:n\in\N\})$ is infinite dimensional and linearly independent from $W$;  furthermore, we have 
  \[\textnormal{span}(\{f_n+h_n:n\in\N\})\setminus\{0\}\subset W+(\widetilde{W}\setminus\{0\})
  \]
  by Lemma \ref{condition on the h_n}; therefore, $(iv_4)$ holds with $V=\textnormal{span}(\{f_n + h_n:n\in\N\})$.
   
\medskip

$\bullet$ Proof of $(ii)\Rightarrow(i)$:
 Assume that $(ii)$ holds. Let $r\in\R$. If $r\ge\sup_{\mathcal{M}^{\tau}(\Omega)} h^\tau$ then $\{\mu\in\mathcal{M}^{\tau}(\Omega): h^\tau(\mu)>r\}=\emptyset$ and $(i)$ holds trivially. Assume that $r<\sup_{\mathcal{M}^{\tau}(\Omega)}h^\tau$ and let $\mu\in\mathcal{M}^{\tau}(\Omega)$ fulfilling $h^\tau(\mu)>r$. The condition 
  $(ii)$ applied to the point $(\mu,h^\tau(\mu))$ of the graph  of $h^\tau$ ensures the existence of a net     $(\mu_i)$  in  $\sE^{\tau}(\Omega)$  converging to $\mu$ and such that   $\lim h^{\tau}(\mu_i)=h^{\tau}(\mu)$; thus,  $h^{\tau}(\mu_i)>r$ eventually, and $(i)$ holds.

\medskip

 $\bullet$ Proof of $(ii)\Rightarrow(v_1)$ with the rate function given in $a) 2)$:
  Assume that $(ii)$ holds. Let $(\nu_\alpha, t_\alpha)$ be a net where  $\nu_\alpha$ is a 
Borel probability measure on $\mathcal{M}(\Omega)$, $t_\alpha>0$ and  $(t_\alpha)$  converges  to zero.
  Assume that
    \[\forall  g\in C(\Omega),\ \ \ \ \ \ \lim_\alpha t_\alpha\log\int_{\mathcal{M}(\Omega)}e^{t_\alpha^{-1}\int_\Omega g(\omega)\mu(d\omega)}\nu_\alpha(d\mu)=Q_{f_0}(g).
    \]
Taking account of Lemma \ref{ergodic-equal-unique-equilibrium}, the condition $(ii)$ together with the above equality shows that  the hypotheses of  Theorem 5.2(b) of \cite{Comman(2009)NON22} hold;  therefore,   the net  $(\nu_{\alpha})$ 
  satisfies a large deviation principle in $\mathcal{M}(\Omega)$ with powers $(t_\alpha)$ and   rate function ${Q_{f_0}^*}_{\mid \mathcal{M}(\Omega)}$. By Lemma \ref{Fenchel-Legendre-transform-of-Q_f},  the condition $(ii)$ is equivalent to the density of   the graph of   ${Q_{f_0}^*}_{\mid \sE^\tau(\Omega)}$   in the graph of  ${Q_{f_0}^*}_{\mid \mathcal{M}^\tau(\Omega)}$, which proves    $(v_1)$.

\medskip

  $\bullet$ Proof of $(v)\Rightarrow(ii)$: Assume that $(v)$ holds.
Since $I$ is convex and for each $g\in C(\Omega)$,  $Q_{f_0}(g)$ coincides with  the value at  $\widehat{g}$ of the  limiting   log-moment generating function associated with $(\nu^\tau_{f_0,\alpha},t^\tau_\alpha)$  by  Lemma \ref{existence-net-generating the pressure}, we have $I={Q_{f_0}^*}_{\mid \mathcal{M}(\Omega)}$ 
    (\cf \S \ref{subsection-Large deviations}) so that $(ii)$ follows from Lemma \ref{Fenchel-Legendre-transform-of-Q_f}.

\medskip

$\bullet$ The second equality in $a)2)$ follows from Lemma  \ref{equality-IS-LS*}.

\medskip

      We have proved 
          \begin{equation}\label{proof-theo-eq20}
  (i)\Leftrightarrow(ii)\Leftrightarrow(v)\Leftrightarrow(v_1)\Rightarrow(iv_4)\Rightarrow(iv)  \textnormal{\ \ \ and property $a) 2)$ concerning $v_1)$.}
\end{equation}
  
\medskip

$\bullet$ Proof of $(iv)\Rightarrow(iii)$ with $D_n$ as in $a) 3)$ for all $n\in\N$: Assume that   $(iv)$ holds.  Let $V$   as in $(iv)$, 
   let  $(f,(g_1,...,g_n),n)\in \Sigma\times V^n\times\N$ and let $\varepsilon>0$.   Since $V$ is  infinite dimensional and  dense in $C(\Omega)$
   there exists  $g_{n+1}\in V\setminus\textnormal{span}\{f,g_1,...,g_n\}$
  such that 
  $\mid\mid f+g_{n+1}\mid\mid<\varepsilon$. Therefore,  $f+g_{n+1}\not\in\textnormal{span}\{g_1,...,g_n\}$ and
    $f+g_{n+1}+\sum_{k=1}^n t_k g_k$  has a unique equilibrium state for all $(t_1,...,t_n)\in\R^n$, which proves     $(iii)$ by putting
     $D_n=V^n$ and $g=f+g_{n+1}$.
     
     \medskip
      
$\bullet$ Proof of $(iii)\Rightarrow(i)$: Assume that   $(iii)$ holds. Let   $n\in\N$ and let $\varepsilon>0$.    The  set  of  all $(g_1,..., g_n)\in C(\Omega)^n$ such that 
 for each $g\in C(\Omega)\setminus\textnormal{span}\{g_1,...,g_n\}$  with $\mid\mid g\mid\mid<\varepsilon$ the function
$g+\sum_{k=1}^n t_k g_k$  has several equilibrium states for  some $(t_1,...,t_n)\in\R^n$,  has empty interior;  since    $g+\sum_{k=1}^n t_k g_k$  has several equilibrium states  if and only if $P_{-h^\tau}$ is not  Gateaux differentiable at $\widehat{g}+\sum_{k=1}^n t_k \widehat{g_k}$ (by (\ref{remark-proof-Phelps-Israel-result-eq60}) in Appendix A),  $(i)$  follows from 
Lemma  \ref{Israel-Phelps-theorem-converse-affine-strenght}.

\medskip

The last two above implications  together with (\ref{proof-theo-eq20}) yield 
\begin{equation}\label{proof-theo-eq30}
(i)\Leftrightarrow(ii)\Leftrightarrow(iii)\Leftrightarrow(iv_4)\Leftrightarrow(iv)\Leftrightarrow(v)\Leftrightarrow(v_1)  \textnormal{\ \ \ and property $a) 3)$ concerning $iv)$ and $iv_4)$}. \end{equation}

\medskip

  $\bullet$ Proof of $(iv)\Rightarrow(vi_2)$ with the rate function given in $a) 2)$: Assume that $(iv)$ holds. Let $V$ as in $(iv)$.  Since $(iv)\Leftrightarrow(v_1)$ by  
  (\ref{proof-theo-eq30}),  and since $f_0$ in $(v_1)$ is an arbitrary element of $C(\Omega)$,  $(v_1)$ holds in particular with    $f_0=f+g$ for all $(f,g)\in\Sigma\times V$; then,  $(vi_2)$ with $V$ follows from the last assertion of  Lemma \ref{Fenchel-Legendre-transform-of-Q_f}.

\medskip

  $\bullet$ Proof of $(vi)\Rightarrow(iv)$: Assume that $(vi)$ holds. Let $V$ as in $(vi)$. The same argument as in the   proof of $(v)\Rightarrow(ii)$ shows that the rate function associated with  $(\nu^\tau_{f+g,\alpha},t^\tau_\alpha)$ is ${Q_{f+g}^*}_{\mid \mathcal{M}(\Omega)}$ so that $(iv)$ with $V$ follows from the last assertion of  Lemma \ref{Fenchel-Legendre-transform-of-Q_f}.
  
  \medskip
  
  We have proved 
  \begin{equation}\label{proof-theo-eq25}
  (iv)\Leftrightarrow(vi)\Leftrightarrow(vi_2)  \textnormal{\ \ \ and property $a) 2)$ concerning $vi_2)$.}
  \end{equation}

\medskip

  $\bullet$ Proof of $(iv)\Rightarrow(vii_3)$ with the rate function given in $a) 2)$:  Assume that $(iv)$ holds. Let $V$ as in $(iv)$.  
 Since $(i)\Leftrightarrow(iv)$ by  
  (\ref{proof-theo-eq30}), and since $(i)$ does not depend on $\Sigma$, one can assume that $\Sigma=W$.
  Let  $(f,g,(f_1,...,f_n),n)\in \Sigma\times V\setminus\{0\}\times E^n\times\N$ and let $(\mu_\alpha,t_\alpha)$ be a net as in $b)3)$.
  By Lemma \ref{uniqueness-equil-implies-diffe-logmoment} (applied with $f+g$ in place of $f$)   the function $L_{f+g,(f_1,...,f_n)}$ is differentiable on $\R^n$ for all $((f_1,...,f_n),n)\in E^n\times\N$ and consequently 
  its Legendre-Fenchel transform $L_{f+g,(f_1,...,f_n)}^*$ is essentially strictly convex (\cf \S \ref{subsection-Convex analysis}).  By  G\"{a}rtner's  theorem (\cf \S \ref{subsection-Large deviations}) the net   $(\mu_\alpha)$ satisfies a large deviation principle in $\R^n$ with powers $(t_\alpha)$ and rate function $L_{f+g,(f_1,...,f_n)}^*$, which  proves $(vii_3)$ with $V$.

\medskip

  $\bullet$ Proof of $(vii)\Rightarrow(viii)$: Assume that $(vii)$ holds. Let $V$ be as in $(vii)$ and let   $(f,g,(f_1,...,f_n), n)\in \Sigma\times V\setminus\{0\}\times E^n\times\N$.  Lemma \ref{existence-net-generating the pressure} implies that $L_{f+g,(f_1,...,f_n)}$ is the limiting  log-moment generating function associated with $((\widehat{f_1},...,\widehat{f_n})[\nu^\tau_{f+g,\alpha}], t^\tau_\alpha)$.  
    Since the rate function governing the large deviation principle  of the net  $((\widehat{f_1},...,\widehat{f_n})[\nu^\tau_{f+g,\alpha}])$ is convex, it must be the Legendre-Fenchel transform $L_{f+g,(f_1,...,f_n)}^*$ of $L_{f+g,(f_1,...,f_n)}$ (\cf \S \ref{subsection-Large deviations})  hence   $(viii)$ holds  with $V$.

\medskip

  $\bullet$ Proof of $(viii)\Rightarrow(ix)$ with the same space $V$ in the premise as in the conclusion: This  follows  from  Lemma \ref{equality-IS-LS*}.

\medskip

  $\bullet$ Proof of $(ix)\Rightarrow(iv)$: Assume that $(ix)$ holds. Let $V$ be  as in $(ix)$. For each   $(f,g,(f_1,...,f_n), n)\in \Sigma\times V\setminus\{0\}\times E^n\times\N$   
 the function   $L_{f+g,(f_1,...,f_n)}^*$ is essentially strictly convex by  Lemma \ref{equality-IS-LS*}  
 hence 
 $L_{f+g,(f_1,...,f_n)}$ is differentiable on  $\R^n$ (\cf \S \ref{subsection-Convex analysis})  and $(iv)$ with $V$ follows from   Lemma \ref{diffe-logmoment-implies-uniqueness-equil}.
 
 \medskip
 
  We have proved
\begin{equation}\label{proof-theo-eq60}
 (iv)\Leftrightarrow(vii_3)\Leftrightarrow(vii)\Leftrightarrow(viii)\Leftrightarrow(ix)  \textnormal{\ \ \ and property $a) 2)$ concerning $vii_3)$.}
 \end{equation}
 
 \medskip

\noindent  The above proof reveals that the  rate function in $(v_1)$, $(vi_2)$,  $(vii_3)$  is given by the Legendre-Fenchel transform of the limiting log-moment generating function associated with the net concerned;  since this  limiting log-moment generating function  is the same in $(v_1)$,  $(vi_2)$,  $(vii_3)$  as in $(v)$, $(vi)$,  $(vii)$,   respectively 
(by Lemma \ref{existence-net-generating the pressure} and the observation before Lemma \ref{uniqueness-equil-implies-diffe-logmoment}) it follows that  property $a) 2)$ concerning $v)$, $vi)$, $vii)$ hold.

Taking into account the foregoing observation and putting together  (\ref{proof-theo-eq30}),   (\ref{proof-theo-eq25}), (\ref{proof-theo-eq60})  proves part $b)$,   property $a) 2)$ and  all the equivalences of part $a)$ as well as   property  $a) 3)$ concerning $iv)$;  furthermore, the above proof  shows that the same vector space $V$ may be used in each of the condition appearing  in  (\ref{proof-theo-eq25}) (resp.   (\ref{proof-theo-eq60})), which gives  property $a) 1)$  and  property   $a) 3)$ in full. The proof of part $a)$ is complete.

 The first assertion of part $c)$ follows  noting that  the above proof  works verbatim replacing  $V\setminus\{0\}$ by $V$ when each element of $\Sigma$ has a unique equilibrium state. 
  
 Let $(iv_0)$, $(vi_0)$, $(vi_{2,0})$, $(vii_0)$, $(vii_{3,0})$,  $(viii_0)$ and $(ix_0)$ denote the conditions obtained respectively from $(iv)$, $(vi)$,  $(vi_2)$,  $(vii)$,  $(vii_3)$, $(viii)$ and $(ix)$  replacing $V\setminus\{0\}$ by $V$. The equivalences \[(iv_0)\Leftrightarrow(vi_0)\Leftrightarrow(vi_{2,0})\Leftrightarrow(vii_0)\Leftrightarrow(viii_{3,0})\Leftrightarrow(viii_0)\Leftrightarrow(ix_0)\] with the same space $V$ in all the above conditions follows 
  replacing $V\setminus\{0\}$ by $V$ in the proof of  (\ref{proof-theo-eq25})  and (\ref{proof-theo-eq60}); this proves  the last  assertion of part $c)$. 
 \end{proof-theo}

  \begin{proof-theo-2}\label{proof-theo-2}
  We can assume that $\Sigma=E=W$ (\cf Remark \ref{remark-(D)-independence of Sigma}). Let $\cN$ be a nonempty subset of $\N\cup\{0\}$. 
  
  First assume that $\cN=\N\cup\{0\}$.  For each $n\in\N\cup\{0\}$ let  $\sP(n)$ denotes the following property: There exists a finite sequence $V_0,...,V_n$ of mutually linearly independent  $\sigma$-compact  infinite dimensional   vector subpaces of $C(\Omega)$  such that $\bigoplus_{j=0}^n V_j$   fulfils   the  conditions $(iv)$,  $(vi)$, $(vii)$, $(viii)$, $(ix)$ of Theorem  \ref{comparison-theorem}.
  We know that $\sP(0)$  holds  by property 1)  in part a) of  Theorem \ref{comparison-theorem}; note that in particular (D) holds. Let $n\in\N\cup\{0\}$ and  assume that $\sP(n)$ holds. Put $W_n=W\oplus\bigoplus_{j=0}^{n} V_n$;  note that $W_n$ is a $\sigma$-compact vector space dense    in $C(\Omega)$ and thus it satisfies the hypotheses  of  Theorem  \ref{comparison-theorem}.  Since   (D) holds,   Lemma \ref{Israel-Phelps-theorem-strenght} ensures the existence of 
a  $\sigma$-compact infinite dimensional vector space $\widetilde{W_n}$ linearly independent from $W_n$ and such that 
$f+g$ has a unique equilibrium state for all $(f,g)\in W_n\times(\widetilde{W_n}\setminus\{0\})$.
(alternatively, the space $\widetilde{W_n}$ can be obtained applying Theorem \ref{comparison-theorem} with $W_n$ in place of $W$, taking into account property $b)4)$). Let  $(f_{n,k})$ be  a Schauder basis  of $C(\Omega)$ included in $W_n$ (\cf \S \ref{subsection-Thermodynamic formalism})  and let $\{h_{n,k}:k\in\N\}$ be  a linearly independent subset  of $\widetilde{W_n}$  fulfilling 
\[\sum_{k=1}^{+\infty}\left(\sup_{f\in C(\Omega), \mid\mid f\mid\mid\le 1}\mid \lambda_{n,k}(f)\mid\right)\mid\mid h_{n,k}\mid\mid<1,\]
where $(\lambda_{n,k}(f))$ denotes the coordinates of $f$ in the basis $(f_{n,k})$.
  Put 
  \[V_{n+1}=\textnormal{span}(\{f_{n,k} + h_{n,k}:k\in\N\}).
\]
  The sequence $(f_{n,k}+h_{n,k})$ is a Schauder basis of $C(\Omega)$ (\cf \S \ref{subsection-Thermodynamic formalism})  hence 
$V_{n+1}$ is dense in $C(\Omega)$; clearly, $V_{n+1}$ is $\sigma$-compact; since $\widetilde{W_n}\cap W_n=\{0\}$ and  $\widetilde{W_n}$ is infinite dimensional, it follows that    $V_{n+1}$ is infinite dimensional and  fulfils $V_{n+1}\cap W_n=\{0\}$; furthermore,  Lemma \ref{condition on the h_n}  applied with
$W_n$ and  $\widetilde{W_n}$ in place of $W$ and $\widetilde{W}$
yields
 \[
 V_{n+1}\setminus\{0\}\subset W_n\oplus(\widetilde{W_n}\setminus\{0\})
 \]
 hence 
  \[W+V_{n+1}\setminus\{0\}\subset W_n\oplus(\widetilde{W_n}\setminus\{0\}).\]
    The recurrence hypothesis at rank $n$  implies that $f+g$ has a unique equilibrium state for all $(f,g)\in W\times \bigoplus_{j=0}^n V_j\setminus\{0\}$. 
  Since
\[\bigoplus_{j=0}^{n+1} V_j\setminus\{0\}=\bigoplus_{j=0}^n V_j\setminus\{0\} \bigcup V_{n+1}\setminus\{0\},\] 
it follows  that  the space $\bigoplus_{j=0}^{n+1} V_j$ fulfils the condition $(iv)$ of Theorem  \ref{comparison-theorem}, which gives 
  $\sP(n+1)$ by property $a)1)$   of Theorem  \ref{comparison-theorem}.  Therefore, $\sP(n)$ holds for all  $n\in\N$ hence the infinite direct sum $\bigoplus_{j=0}^\infty V_j$ fulfils 
 the condition $(iv)$ of  Theorem  \ref{comparison-theorem}
  since each element of $\bigoplus_{j=0}^\infty V_j$ (resp. $\bigoplus_{j=0}^\infty V_j\setminus\{0\}$) belongs to  $\bigoplus_{j=0}^n V_j$ (resp. $\bigoplus_{j=0}^n V_j\setminus\{0\}$)) for some $n\in\N$;  
  consequently,   property $a)1)$  of Theorem  \ref{comparison-theorem} implies that  the conclusion of part  $a)$ holds when $\cN=\N\cup\{0\}$.
  
  Let  $\cN\neq\N\cup\{0\}$.   Note that  $W+\bigoplus_{n\in\N\cup\{0\}\setminus\cN} V_n$ is a $\sigma$-compact vector space  dense in $C(\Omega)$ and thus  it fulfils the hypotheses of Theorem  \ref{comparison-theorem}. Since 
   \[(W+\bigoplus_{n\in\N\cup\{0\}\setminus\cN} V_n)+ (\bigoplus_{n\in\cN} V_n\setminus\{0\})\subset W+\bigoplus_{n\in\N\cup\{0\}} V_n\setminus\{0\},\]
   the preceding case (when $\cN=\N\cup\{0\}$) implies that $\bigoplus_{n\in\cN} V_n$ fulfils 
   the condition  $(iv)$  of Theorem  \ref{comparison-theorem} applied with $W+\bigoplus_{n\in\N\cup\{0\}\setminus\cN} V_n$ in place of $W$;    the conclusion of part  $a)$ when $\cN\neq\N\cup\{0\}$  follows by property $a)1)$  of Theorem  \ref{comparison-theorem} (applied with $W+\bigoplus_{n\in\N\cup\{0\}\setminus\cN} V_n$ in place of $W$). We have proved part $a)$ and the first assertion of part $b)$.
   
    If   furthermore   $V_0\cap W=\{0\}$, then the above proof works verbatim with $W$ in place of $W_0$ and 
   taking $\widetilde{W_0}=V_0$ (indeed, the choice of $W+V_0$  and not just  $W$   as general   definition  of $W_0$    is  made in order to ensure that $W_0\cap\widetilde{W_0}=\{0\}$); if   $W$ contains an element  admitting several equilibrium states,  then  $V_0\cap W=\{0\}$ by the condition  $(iv)$ of Theorem  \ref{comparison-theorem}; this proves the last assertion of part $b)$. 
     
     For each $n\in\N\cup\{0\}$ let $\cV_{n+1}$ and $W_n$ be the  spaces defined  in  part $c)$.  Put $\cV_0=V_0$ and for each   $n\in\N\cup\{0\}$ let $\sQ(n)$ denote the following inclusion: 
   \[W+\sum_{j=0}^n \widetilde{V}_j\subset W+ \bigoplus_{j=0}^n \cV_j.\]
  By definition $\sQ(0)$ holds. Let $n\in\N\cup\{0\}$ and  assume that $\sQ(n)$ holds.
  Since by definition  \[\widetilde{V}_{n+1}\subset W+ \bigoplus_{j=0}^n \widetilde{V}_j+\cV_{n+1},\]
         we have 
    \[W+\sum_{j=0}^{n+1}\widetilde{V}_j\subset W+ \bigoplus_{j=0}^n \widetilde{V}_j+\cV_{n+1};\]
   the recurrence hypothesis at rank $n$ together with the above inclusion  yields
   \[W+\sum_{j=0}^{n+1}\widetilde{V}_j\subset  W+ \bigoplus_{j=0}^{n+1} \cV_j,\]
      which proves $\sQ(n+1)$. Therefore, $\sQ(n)$ holds for all $n\in\N\cup\{0\}$. 
      Let $n\in\N\cup\{0\}$. Since $\cV_{k}\subset\bigoplus_{j=m_{k}}^{j=m_{k+1}} V_j$ for all $k\in\N$ with  $\bigoplus_{j=1}^\infty V_j$   as in part $a)$, it follows that  $f+g$ has a unique equilibrium state for all $(f,g)\in W\times (\bigoplus_{j=0}^{n+1} \cV_j\setminus\{0\})$ and in particular for all  $(f,g)\in W\times (\bigoplus_{j=0}^n\cV_j\oplus\cV_{n+1}\setminus\{0\})$; this property  together with $\sQ(n)$ shows that all the hypotheses of part $b)$ hold replacing $V_n$ by $\widetilde{V}_n$ and taking $\widetilde{W_n}=\cV_{n+1}$ for all $n\in\N\cup\{0\}$, which gives all the conclusions of part $c)$.
     \end{proof-theo-2}

\begin{proof-corollary}\label{proof-corollary}
Let $(v_f)$ denote the statement obtained from $(v)$ with the change described in  the  property $2)$  with $f\in V$. 
Let $(ii_\sigma)$ (resp. $(iii_\sigma)$, $(iv_\sigma)$, $(v_\sigma)$, $(v_{\sigma,f})$) denote the statement obtained from  $(ii)$ (resp. $(iii)$, $(iv)$, $(v)$, $(v_f)$)  assuming furthermore that $V$ is $\sigma$-compact. 

The condition 
        $(iii_\sigma)$ is equivalent to  the condition $(viii)$ of Theorem \ref{comparison-theorem}  with  $\Sigma=\{0\}$,   $V=W=E$; furthermore, the equalities  $V=W=E$ allows us to replace $V\setminus\{0\}$ by $V$. 
        Therefore,    the last assertion of part c) of  Theorem \ref{comparison-theorem}  yields
        \begin{equation}\label{proof-corollary-eq20}
 (i)\Leftrightarrow(ii_\sigma)\Leftrightarrow(iii_\sigma)\Leftrightarrow(iv_\sigma)\Leftrightarrow(v_{\sigma,f})\Leftrightarrow(v_\sigma)
 \end{equation}
 and the validity of the property $2)$;  furthermore, by the property $a)1)$  of  Theorem \ref{comparison-theorem},   
  if a vector space  fulfils one of the   conditions $(ii_\sigma)$, $(iii_\sigma)$, $(iv_\sigma)$, $(v_{\sigma,f})$, $(v_\sigma)$ then it  fulfils all of them.
 Since   $(ii_\sigma)$ (resp. $(iii_\sigma)$, $(iv_\sigma)$, $(v_{\sigma,f})$) is clearly equivalent to   $(ii)$ (resp. $(iii)$, $(iv)$, $(v_f$)) we have proved all the assertions of Corollary \ref{criterium-dense-vector-space-unique-equilibrium} except the property  $1)$ (note that the property $2)$ does not depend on the $\sigma$-compactness of $V$).
 
 If an infinite dimensional   vector  space $V$  dense in $C(\Omega)$ fulfils one of the  conditions $(ii)$, $(iii)$, $(iv)$, $(v)$,  then  for each $\sigma$-compact vector space $V_0$  dense in $V$ the hypotheses of  Theorem \ref{use-Schauder-basis-to-get-V} hold with $E=\Sigma=W=V_0$ (\cf Remark \ref{remark-Corollary-implies-Theorem-V=W}); therefore, property   $1)$  follows from part $a)$ of Theorem \ref{use-Schauder-basis-to-get-V}.
\end{proof-corollary}

\appendix

\section{}

 The paper  \cite{Israel-Phelps-84-MAGHSCAN-54} deals with 
 a nonempty convex compact subset $K$ of some Hausdorff locally convex  real topological vector  space,  the set $A(K)$ of all real-valued  affine continuous  functions on $K$ endowed with the uniform topology,  a  convex bounded non-positive-valued lower semi-continuous function $l$ on $K$ and the function $P_l$ on $A(K)$ associated to $l$  in the following way:
 \[\forall a\in A(K),\ \ \ \ \ \ \ P_l(a)=\sup_{x\in K}\{a(x)-l(x)\};\]
 the map $P_l$   is real-valued  convex   and continuous  (\cite{Israel-Phelps-84-MAGHSCAN-54}, Proposition 2.3). 
 When  $K=\mathcal{M}^{\tau}(\Omega)$  it is known that the map  
$C(\Omega)\ni g\mapsto \widehat{g}\in A(\mathcal{M}^{\tau}(\Omega))$
 is surjective
  (\cite{Israel-Phelps-84-MAGHSCAN-54}, Proposition 2.1); when furthermore $l=-h^\tau$ we   have 
 \begin{equation}\label{remark-proof-Phelps-Israel-result-eq60}
\forall (f,g,t)\in C(\Omega)^2\times\R,\ \ \ \ \ \ \ \ P_{-h^\tau}(\widehat{f+tg})=P_{-h^\tau}(\widehat{f}+t\widehat{g}))=P^\tau(f+tg);
 \end{equation} 
 in particular, $dP^\tau(f;g)$ coincides with the  directional derivative  of  $P_{-h^\tau}$ at $\widehat{f}$ in the direction $\widehat{g}$.
 Since  $h^\tau$ is affine and ${\sE}^\tau(\Omega)$ is the set of  extreme points of $K$,  Theorem 3.2 and 
 Theorem 3.3   of \cite{Israel-Phelps-84-MAGHSCAN-54}  take respectively  the following  forms:  
 
 \begin{theorem}\textnormal{\textbf{(Israel-Phelps)}}\label{Israel-Phelps-theorem}
 Assume that property (i) of Theorem  \ref{comparison-theorem}  holds. Then, for each $\sigma$-compact subset $S$ of $A(\mathcal{M}^{\tau}(\Omega))$ and  for each $n\in\N$,  the set $G_n$ of all elements $(a_1,...,a_n)\in A(\mathcal{M}^{\tau}(\Omega))^n$
 such that $P_{-h^\tau}$ is Gateaux differentiable at  $b+\sum_{k=1}^n t_k a_k$  for all  $(b,(t_1,..., t_n))\in S\times\R^n\setminus\{0\}$, is a  $G_\delta$  set dense in $A(\mathcal{M}^{\tau}(\Omega))^n$.
\end{theorem}

 \begin{theorem}\textnormal{\textbf{(Israel-Phelps)}}\label{Israel-Phelps-theorem-converse-affine}
If  property (i) of Theorem  \ref{comparison-theorem} does not hold, then there exists  $n\in\N$ such that  the set of 
 all  $(a_1,...,a_n)\in {A(\mathcal{M}^{\tau}(\Omega))}^n$ 
 for which $P_{-h^\tau}$  is not Gateaux differentiable  at  $\sum_{k=1}^n t_k a_k$ for some  $(t_1,...,t_n)\in\R^n\setminus\{0\}$, has nonempty interior. 
\end{theorem}


\begin{thebibliography}{100}


\bibitem{Beardon-91}
A. F. Beardon, \textit{Iteration of rational functions},  Complex analytic dynamical systems, Graduate Texts in Mathematics, Vol. 132, Springer-Verlag 1991.
 
 

\bibitem{Comman_Rivera-Letelier(2010)ETDS31}
 H. Comman, J. Rivera-Letelier. Large deviation principles for non-uniformly hyperbolic rational maps.
  \textit{Ergodic Theory and Dynamical Systems 31, No. 2 (2011)  321-349.}


\bibitem{Comman(2009)NON22}
H. Comman. Strengthened large deviations for rational maps and full shifts, with
  unified proof. \textit{Nonlinearity 22 (6) (2009) 1413-1429}.
  
  
\bibitem{Comman(2007)STAPRO77}
H. Comman. Variational form of the large deviation functional.
\textit{Statistics and Probability Letters 77 (2007), No. 9,
931-936.}





\bibitem{com-TAMS-03}
H. Comman. Criteria for large deviations. \textit{Trans. Amer.
Math. Soc. 355 (2003), no. 7,  2905-2923.}


  
\bibitem{Dembo-Zeitouni}
A. Dembo,  O. Zeitouni. \textit{Large deviations techniques and
applications}, Second edition, Springer 1998.




\bibitem{Ekeland_Teman}
I. Ekeland,  R. Teman. \textit{Convex analysis and variational problems}, North-Holland,  1976.



\bibitem{Gartner-TeorVerojatnostPrimenen-77}
Jurgen G{\"a}rtner. On large deviations from an invariant measure. 
\textit{Teor. Verojatnost.  Primenen, 22 (1) (1977) 27-42}.



\bibitem{Israel-Phelps-84-MAGHSCAN-54}
R. B. Israel,  R. R. Phelps.  Some convexity questions arising in statistical mechanics. \textit{Math. Scand. 54 (1984) 133-156}.


\bibitem{Israel-86-CMP-106}
R. B. Israel. Generic triviality of phase diagrams in spaces of long-range interactions. 
\textit{Commun. Math. Phys. 106 (1986) 459-466}





\bibitem{Kelley-91}
J. L. Kelley. \textit{General topology},  Springer 1991.


\bibitem{Kifer-TAMS-90}
Y. Kifer. Large deviations in dynamical systems and stochastic processes.
\textit{Trans. Amer. Math. Soc. 321 (1990) 505-524}.




\bibitem{Ljubich-83-ETDS-3}
M. J.  Ljubich. Entropy properties of rational endomorphisms of the Riemann sphere.
\textit{Ergodic Theory Dynam. Systems 3 (1983) 351-385}.


\bibitem{Phelps_Dynamics and Randomness(2002)Santiago}
R. R. Phelps. Unique equilibrium states.  \textit{Dynamics and Randomness, 219-225 (2002), A. Mass et al. (Eds.)}


\bibitem{Przytycki_Rivera-Letelier_Smirnov(2003)InventMath151}
F. Przytycki, R. Rivera-Letelier, S. Smirnov. Equivalence and topological invariance of conditions for non-uniform hyperbolicity in the iteration of rational maps. \textit{Invent. Math. 151 (1) (2003) 29-63}.


\bibitem{Rockafellar-70}
R. G. Rockafeller, \textit{Convex analysis}, Princeton University
Press 1970.



\bibitem{Ruelle-78}
D. Ruelle. \textit{Thermodynamic formalism},  Addison-Wesley 1978.



\bibitem{Semeradi-82}
Z. Semeradi. \textit{Schauder basis in Banach spaces of continuous functions}, Springer-Verlag 1982.


\bibitem{Sokal-(82)-CMP-86}
A. D. Sokal. More surprises in the general theory of lattice systems. \textit{Commun. Math. Phys. 86 (1982) 327-336}.

\end{thebibliography}
\end{document}